\documentclass[10pt,1p]{elsarticle}
\usepackage{hyperref}
\usepackage{capt-of}
\usepackage{cancel}
\usepackage{subcaption}
\usepackage{epstopdf}
\usepackage{graphicx}
\usepackage[table]{xcolor}
\usepackage{wrapfig}
\usepackage{import}
\usepackage{calc}
\newcommand{\executeiffilenewer}[3]{%
\ifnum\pdfstrcmp{\pdffilemoddate{#1}}%
{\pdffilemoddate{#2}}>0%
{\immediate\write18{#3}}\fi%
}

\usepackage{csquotes}% Recommended
\usepackage{blindtext}
\usepackage{booktabs}
\usepackage{varwidth}
\usepackage[bottom]{footmisc}
\usepackage{csquotes}
\usepackage{booktabs} % For professional looking tables
\usepackage{multirow}

\usepackage[english]{babel}% Recommended
\usepackage{fixltx2e}

\usepackage[inline]{enumitem}
\usepackage{amsfonts}
\usepackage{amsmath}
\usepackage{mathtools}
\usepackage{amssymb}
\usepackage{commath}
\usepackage{amsthm}
\usepackage{siunitx} % use this package module for SI units
\numberwithin{equation}{section}
\DeclareSIUnit\gauss{G}

\usepackage{xparse}
\NewDocumentCommand\angRange{O{} m m}{\SIrange[parse-numbers=false, 
  #1]{\ang[parse-numbers=true]{#2}}{\ang[parse-numbers=true]{#3}}{}}

\usepackage{relsize}
\usepackage{bm}
\newtheorem{theorem}{Theorem}[section]
\theoremstyle{remark}
\newtheorem*{remark}{Remark}
\theoremstyle{definition}

\renewcommand{\vec}[1]{\mathbf{#1}}
\newcommand{\e}{{\mathrm{e}}}
\DeclarePairedDelimiter{\diagfences}{(}{)}
\newcommand{\diag}{\operatorname{diag}\diagfences}
\newcommand{\bigO}[1]{\mathop{}\!O{\left(#1\right)}}

\usepackage{algorithm}
\usepackage[noend]{algpseudocode}
\usepackage{units}
\newcommand{\diff}[1]{\mathrm{d#1}}
\makeatletter
\def\BState{\State\hskip-\ALG@thistlm}
\makeatother
\makeatletter
\newcommand{\opnorm}{\@ifstar\@opnorms\@opnorm}
\newcommand{\@opnorms}[1]{%
	\left|\mkern-1.5mu\left|\mkern-1.5mu\left|
	#1
	\right|\mkern-1.5mu\right|\mkern-1.5mu\right|
}
\newcommand{\@opnorm}[2][]{%
	\mathopen{#1|\mkern-1.5mu#1|\mkern-1.5mu#1|}
	#2
	\mathclose{#1|\mkern-1.5mu#1|\mkern-1.5mu#1|}
}
\makeatother

\algnewcommand{\Inputs}[1]{%
  \State \textbf{Inputs:}
  \Statex \hspace*{\algorithmicindent}\parbox[t]{.8\linewidth}{\raggedright #1}
}
\algnewcommand{\Initialize}[1]{%
  \State \textbf{Initialize:}
  \Statex \hspace*{\algorithmicindent}\parbox[t]{.8\linewidth}{\raggedright #1}
}
\makeatletter
\def\@author#1{\g@addto@macro\elsauthors{\normalsize%
		\def\baselinestretch{1}%
		\upshape\authorsep#1\unskip\textsuperscript{%
			\ifx\@fnmark\@empty\else\unskip\sep\@fnmark\let\sep=,\fi
			\ifx\@corref\@empty\else\unskip\sep\@corref\let\sep=,\fi
		}%
		\def\authorsep{\unskip,\space}%
		\global\let\@fnmark\@empty
		\global\let\@corref\@empty  %% Added
		\global\let\sep\@empty}%
	\@eadauthor={#1}
}
\makeatother
\usepackage{microtype}

\usepackage{enumitem}

\usepackage{titlecaps}
\usepackage{lipsum}
\usepackage{appendix}
\usepackage{cleveref}
\crefname{equation}{eqn.}{equations}
\crefname{problem}{problem}{problems}
\crefname{expression}{expression}{expressions}
\crefname{appsec}{appendix}{appendices}
\crefname{lemma}{lemma}{lemmas}
\crefname{figure}{fig.}{figures}
\crefname{algorithm}{Algorithm}{Algorithms}
\crefname{section}{Sec.}{Sections}

\addcontentsline{toc}{section}{References}
\graphicspath{{./Apnum_Figures/}}
\journal{Applied Numerical Mathematics}

\begin{document}

\begin{frontmatter}



\title{Numerical Simulation of Bloch Equations for \\ Dynamic Magnetic Resonance Imaging}


%\author{Arijit Hazra}
%\author[rvt]{Arijit~Hazra\corref{cor1}\fnref{fn1}}
\author{Arijit~Hazra\corref{cor1}}
%\ead{arijit.hazra@stud.uni-goettingen.de}
\ead{arijit.hazra@stud.uni-goettingen.de}
\author{Gert~Lube\corref{cor1}}
\ead{lube@math.uni-goettingen.de}
\cortext[cor1]{Corresponding authors}
\author{Hans-Georg~Raumer}
\ead{hansgeorg.raumer@stud.uni-goettingen.de}
\address{Institute for Numerical and Applied Mathematics,\\ 
         Georg-August-University G\"{o}ttingen, D-37083, Germany}

\begin{abstract}
Magnetic Resonance Imaging (MRI) is a widely applied non-invasive imaging modality based 
on non-ionizing radiation which gives excellent images and soft tissue contrast of living tissues.
We consider the modified Bloch problem as a model of MRI for flowing spins in an incompressible 
flow field. After establishing the well-posedness of the corresponding evolution problem, we 
analyze its spatial 
semi-discretization using discontinuous Galerkin methods. The high frequency time evolution 
requires a proper explicit and adaptive temporal discretization. The applicability of the 
approach is shown for basic examples. 
\end{abstract}


\end{frontmatter}


\section{Bloch Model for Magnetic Resonance Imaging}\label{sec:MRI}
Magnetic Resonance Imaging (MRI) is a non-invasive imaging modality based on non-ionizing radiation 
\cite{nishimura1996principles}.
It gives excellent images and soft tissue contrast of living tissues. 
%Due to the abundance of in living tissue and high sensitivity of the Hydrogen atom, it is studied 
%in an MRI experiment. 
During the experiment, the object to be studied is placed in a static magnetic field of high strength $B_0$. 
This induces a macroscopic nuclear magnetization $\vec M_0$ in the direction of magnetic field, known 
as the equilibrium magnetization. 
The direction of equilibrium magnetization is called the longitudinal direction, generally denoted by
the $z$-axis as in \Cref {fig:MRPhenomena}.  

In order to get a response from the object, the equilibrium magnetization is perturbed by applying a 
short radio-frequency (RF) pulse $\vec B_1$ in the transverse plane (denoted by the $xy$-plane in 
\Cref{fig:MRPhenomena}) with the excitation carrier frequency of the RF pulse equal to the Larmor 
frequency of {protons}. As a result, the equilibrium magnetization is flipped from its initial position 
towards the transverse plane. The perturbation of equilibrium magnetization depends on the duration 
and magnitude of the RF field. After the RF pulse is switched off, the magnetization precesses towards 
its equilibrium position. This process is known as relaxation. The relaxation of magnetization is 
governed by two time constants -
%two phenomenologically determined parameters-
 spin-lattice time relaxation $T_1$, spin-spin 
time relaxation $T_2$, as illustrated in \Cref{fig:MRPhenomena}.

\begin{figure}[t]
  \centering
  \includegraphics[width=0.90\columnwidth]{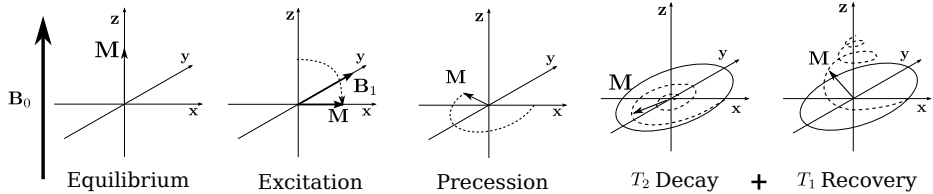}
    \caption{Schematic of a pulsed MR experiment on net Magnetization $\mathbf{M}$. Left: $\mathbf{M}$ 
       align along the static magnetic field  $B_0$. Second from left: RF excitation perturbs the magnetization 
       from the longitudinal direction. Middle: Precession of magnetization as a result of,
       Second from right: $T_2$ decay. Right: $T_1$ recovery.}
    \label{fig:MRPhenomena}
\end{figure}
%The time evolution of magnetization due to the combined effect of the static magnetic field, the RF 
%excitation field and the relaxations are illustrated schematically in \Cref{fig:MRPhenomena}.

%The time evolution of magnetization due to the combined effects of the static magnetic field, the RF 
%excitation field and the relaxations is described by the \textit{Bloch equation} \cite{Bloch1946nuclear}, and 
%for the magnetic field $\vec{M}(t,\vec r)=(M_x,M_y,M_z)^T(t,\vec r)$ with $\vec r=(x,y,z)^T$ it is given by :
%\begin{align}
%  \dod{\vec{M}}{t}=\gamma \vec{M}\times (\vec {B_0}+\vec{B_1}) + 
%  \frac{M_0-M_z}{T_1} \hat{e}_z
%  - \frac{M_x}{T_2}\hat{e}_x - \frac{M_y}{T_2}\hat{e}_y.
%\end{align}
%%In an MRI experiment, emitted energy from rotating magnetizations is converted into a electric signal 
%%which is used further for generating images. 
%
% 
%
%%\subsection{Bloch Equations for Magnetic Resonance Imaging}\label{sec:blochstatic}
%Emitted energy due to the precession of the transverse magnetization $M_{xy}=M_x+i M_y$ is converted into 
%an electric signal which is used further for generating images.

{ On top of that magnetic field gradients need to be applied to obtain cross-sectional images. The location and thickness of the slice is determined by applying a slice-selection gradient $G_z$. After that, the object is spatially encoded via the application 
of additional magnetic field gradients $G_x, G_y$ in the transversal directions.
The magnitude and time of application of magnetic gradients $\vec G$ depend on the experimental 
requirements, see \Cref{sec:experiments} for the description of a typical pulse sequence.

The \textit{Bloch equation} \cite{Bloch1946nuclear} for MRI combines all of the above components: static magnetic field $\mathbf{B_0}$, time-dependent magnetic gradients $\vec G(t)$ and the RF pulse $\vec B_{1}(t)=(B_x(t), B_y(t), 0)^T$ and is given by
\begin{align}
\dod{\vec{M}}{t}&=\gamma \vec{M}\times \vec B + 
\frac{M_0-M_z}{T_1} \hat{e}_z - \frac{M_x}{T_2}\hat{e}_x - \frac{M_y}{T_2}\hat{e}_y,\\
\vec B(t,\vec r) &= (B_{x}(t), B_{y}(t), B_z(t,\vec r))^T, \quad
B_z(t,\vec r) := B_0+\vec G(t)\cdot \vec r \nonumber.
\end{align}
%{\color{green} Also, 
%magnetic field inhomogeneities $\Delta B$ in the MRI system must be included in the most general 
%form of the Bloch equation for MRI which is given by}
%
%{\color{green}
%\begin{align*}
%	\dod{\vec{M}}{t}&=\gamma \vec{M}\times \vec B + 
%	\frac{M_0-M_z}{T_1} \hat{e}_z - \frac{M_x}{T_2}\hat{e}_x - \frac{M_y}{T_2}\hat{e}_y,\\
%	\vec B(t,\vec r) &= (B_{x}(t), B_{y}(t), B_z(t,\vec r))^T, \quad
%          B_z(t,\vec r) := B_0+\vec G(t)\cdot \vec r + \Delta B .
%\end{align*}
%}
Emitted energy due to the precession of the magnetizations is converted into an electric signal in the receiver coil of an MRI system which is manipulated further for image reconstruction.
%For image reconstruction, the electric signal induced by the precessing transverse magnetization is 
%acquired in the receiver coil of an MRI system.
The received signal can be expressed as 
\begin{align}\label{eq:recsig}
s_r(t)\varpropto \int_{\Omega}M_{xy}(\vec r,t)\e^{-i\gamma B_0t} \e^{-i\gamma\int_0^t (xG_x+yG_y) d\tau} d\vec{r}
\end{align}
where $M_{xy} = M_x + i M_y$.
The received signals are typically demodulated in frequency by $\gamma B_0$ using phase-sensitive detection before 
being used for image reconstruction. The resultant signal expression after demodulation is
$s_d(t) = s_r(t) \e^{i\gamma B_0 t}$ \cite{nishimura1996principles}.

The demodulated signal corresponds to the solution of the Bloch equation in a frame rotating clockwise 
about $z$-axis with an angular frequency $ \gamma B_0$, i.e, the Larmor frequency of the isocenter of the object. 
The Bloch equation in the rotating frame conceptually simplifies the RF excitation effect in MRI and eliminates the static $\vec B_0$ field from the expression of the external magnetic field and is given by
\begin{subequations}\label{eq:bloch}
  \begin{align}\label{eq:bloch1}
    \dod{\vec {M}'}{t} &= \gamma\vec {M}'\times \vec B_{\text{eff}}+
    \frac{(M_0-M_{z})\hat {e}_z}{T_1}- \frac{M_{x'} \hat {e}_{x'}+M_{y'} \hat{e}_{y'}}{T_2}\\    \label{eq:bloch2} 
    &= \begin{bmatrix}
      -\frac{1}{T_2} & \gamma B_{z'} & -\gamma B_{y'}  \\
      -\gamma B_{z'} & -\frac{1}{T_2}  & \gamma B_{x'} \\
      \gamma B_{y'} & -\gamma B_{x'}  & -\frac{1}{T_1} 
    \end{bmatrix} \begin{pmatrix}M_{x'} \\M_{y'} \\M_z \end{pmatrix}
    +\begin{pmatrix}0\\0\\ \frac{M_0}{T_1}\end{pmatrix},
  \end{align}
\end{subequations}
where $\mathbf{M}'(t,\vec r) = (M_{x'}, M_{y'}, M_{z})^T(t,\vec r)$ represents the magnetization in the rotating frame 
and the effective magnetic field is given by $\mathbf{B}_{\textnormal{eff}}(t,\vec r)=(B_{x'}, B_{y'}, B_{z'})^T(t,\vec r)$; $B_{z'}(t,\vec r) := G(t) \cdot \vec r$.}

%thanks to 
%$B_0- \frac{\omega_\text{rf}}{\gamma}=0$ for $\omega_\text{rf}=\omega_0$. 
%    \intertext{where}  \vec B_{\text{eff}}&=B_{x}(t) \hat {e}_{x'}+B_{y}(t) \hat {e}_{y'}+ 
%  \underbrace{\cancelto{0}{(B_0- \frac{\omega_\text{rf}}{\gamma})}}_{\omega_\text{rf}=\omega_0}\hat e_z
%  +(\vec G(\vec r,t)\cdot\vec r  + \Delta B ) \hat {e}_{z}. \label{eq:effmagnetic}
%
%\subsection{Bloch Equations for Flowing Objects} \label{sec:blochflow}

In order to study the effect of fluid flow in an MRI experiment, the transport of magnetizations 
due to flow field $\vec u(t,\vec r)$ must be taken into account and is modeled by the modified Bloch equation \cite{Stejskal1965}:
\begin{align}\label{eq:blochflowlab}
  \dpd{\mathbf{M}}{t}+ (\mathbf u \cdot \boldsymbol\nabla)\mathbf{M}
    &= \gamma\mathbf{M} \times \mathbf {B}+
    \frac{(M_0-M_z)\hat e_z}{T_1}- \frac{M_x \hat e_x+M_y \hat e_y}{T_2}.
\end{align}
There will be no diffusion term as effect of diffusion is negligible, but see \Cref{sec:Well-posedness}.

Due to signal demodulation, the signal acquired from the flowing object in an
MRI experiment is equivalent to an equation where
magnetizations and the magnetic field ($\vec B_{\text{eff}}$) are written in the
rotating frame with velocity $\vec u(t,\vec r)$ kept in laboratory frame \cite{Lorthois2005numerical}:
\begin{align}\label{eq:blochflow}
  \dpd{\mathbf{M}'}{t}+(\mathbf u \cdot
  \boldsymbol\nabla)\mathbf{M}'=\gamma\mathbf{M}' \times \mathbf
  {B_{\text{eff}}}+\frac{(M_0-M_{z})\hat e_z}{T_1}- \frac{M_{x^\prime} \hat
    e_{x^\prime}+M_{y^\prime} \hat e_{y\prime}}{T_2}.
\end{align}
For notational simplicity the $'$ will be omitted in later sections. Hereafter, magnetization $\vec M $ 
will always be in the rotating frame and the velocity in the laboratory frame.

{ In the past, Bloch equation based MRI simulators were developed for variety of research 
directions e.g. optimizing MR sequences, artifact detection, testing image reconstruction techniques, 
design of specialized RF pulses and educational purposes. Also, multiple utilities of numerical 
simulations have been combined to produce a few general purpose MRI simulators e.g., 
\cite{Bittoun1984computer,Stoecker2010high,Hargreaves2004bloch,Benoit-Cattin2005simri,Xanthis2014mrisimul}. 
Accurate simulation of this initial-value problem is still challenging for the following reasons: very 
tiny time steps, sufficiently fine spatial resolution, and non-smooth data (e.g. gradient field $\vec G$). 
To overcome this difficulty, the execution speeds were improved using parallelization with message passing 
interface (MPI), e.g. \cite{Benoit-Cattin2005simri,Stoecker2010high}, or GPU e.g. \cite{Xanthis2014mrisimul}.
There are two prevailing techniques for the development of the simulator (i) a semi-analytical technique
 based on operator splitting as used, e.g.,  in \cite{Benoit-Cattin2005simri, Xanthis2014mrisimul} 
(ii) coupled approach with fast ODE-solvers \cite{Stoecker2010high}.

Similarly, there are many numerical studies on the influence of flow on MRI e.g., 
\cite{Yuan1987solution,Jou1996calculation,Lorthois2005numerical,Jurczuk2013computational}. Modified 
Bloch model was solved using multiple numerical strategies previously e.g., Jou et al. 
\cite{Jou1996calculation}, Lorthois et al. \cite{Lorthois2005numerical} first solved the flow field in 
a computational mesh using finite volume method (FVM) softwares and then studied the effect of flow 
on magnetization using finite difference method (FDM). 
Jurczuk et al. \cite{Jurczuk2013computational} solved \ref{eq:blochflow} by splitting the transport 
and the MR terms. In their work, the magnetization transport was calculated using lattice-Boltzmann 
method (LBM) and the reaction part was calculated using operator splitting techniques.

However, the previous studies lack a proper analysis of the Bloch model for flowing objects. 
\Cref{sec:Well-posedness} is devoted to prove the well-posedness of the modified Bloch model. 
The spatial semi-discretization with DG-methods is presented in \Cref{sec:semidisc} which, to 
the best of our knowledge, is also an addition to the literature related to Bloch model. 
This section was followed by consideration of the temporal discretization in \Cref{sec:temporal}. 
Verifications for simple problems for static and flowing objects are presented
in \Cref{sec:experiments}.}

%The paper is organized as follows: In Sec.~\ref{sec:Well-posedness} we prove the well-posedness
%of the Bloch model. Then, in Sec.~\ref{sec:semidisc}, the spatial semidiscretization with
%DG-methods is analyzed, followed by consideration of the temporal discretization in 
%Sec.~\ref{sec:temporal}. Some numerical results are presented in Sec.~\ref{sec:experiments}.

\section{Well-posedeness of the Bloch Model} \label{sec:Well-posedness}
Let $\Omega \subset \mathbb{R}^3$ be the flow domain with piecewise smooth Lipschitz boundary 
$\Gamma$ and outer normal $\vec n$. The incompressible flow field $\vec u (t,\vec r)$ introduces the 
splitting $\Gamma = \Gamma_{-}\bigcup \Gamma_{+} \bigcup \Gamma_{0}$ where
$\Gamma_{-}=\{\vec r \in \Gamma | \vec u \cdot \vec n <0 \}, \; 
\Gamma_{+}=\{\vec r \in \Gamma | \vec u \cdot \vec n > 0 \} \; \textnormal{and} \; 
\Gamma_{0}=\{\vec r \in \Gamma | \vec u \cdot \vec n = 0 \}$ 
represent the inflow boundary, outflow boundary and solid wall, respectively.
We assume that inflow and outflow are separated, i.e. $\mbox{dist}(\Gamma_-,\Gamma_+):=
\min_{(P,Q) \in \Gamma_- \times \Gamma_+}|P-Q| >0$.

Problem \labelcref{eq:blochflow} can be rewritten using 
$\vec B_{\textnormal{eff}}=\vec B = \begin{pmatrix} B_x,B_y,B_z \end{pmatrix}^T$, 
the relaxation time diagonal matrix $D = \diag{\frac{1}{T_2},\frac{1}{T_2}, \frac{1}{T_1}}$ 
and the constant source term $\vec f=\begin{pmatrix} 0, 0, \frac{M_0}{T_1}\end{pmatrix}^T$
along with appropriate boundary and initial conditions as
\begin{subequations}\label{eq:genblochflow}
	\begin{align}\label{eq:genblochflow1}
	\dpd{\vec M}{t} + (\mathbf u \cdot
	\boldsymbol\nabla)\mathbf{M} + \gamma \vec B \times \vec M + D \vec M &= \vec f, 
	\qquad (t,\vec r) \in [0,T] \times \Omega,\\
	\vec M &= \vec M_\Gamma, \quad (t,\vec r) \in [0,T] \times \Gamma_{-}, \label{eq:genblochflowb}\\ 
	\vec M &= \vec M^0 , \quad (t,\vec r) \in \{0\}\times \Omega. \label{eq:genblochflowi}
	\end{align}
\end{subequations}

Consider the space ${\mathbf H}:=[L^2(\Omega)]^3$ with inner product $(\vec u,\vec v)_{\mathbf H}:= 
\int_\Omega \vec u \cdot \vec v~ \dif{\vec r}$ and norm $\norm{\vec v}_{\mathbf H}:= \sqrt{(\vec v,\vec v)}$. 
Moreover, define
\begin{align}\label{eq:graphspace}
\mathbf{X} =\{ \vec N \in \mathbf{H}:~ (\mathbf u \cdot
\boldsymbol\nabla)\vec N\in {\mathbf H}\}
\end{align} 
with graph norm
\begin{align}\label{eq:graphnorm}
\| \vec N\|_{\mathbf X} := \left( \| (\vec u \cdot \nabla )\vec N\|^2_{\mathbf H} + 
\| \vec N\|^2_{\mathbf H} \right)^{\frac12}.
\end{align}
Multiplying \labelcref{eq:genblochflow1} by an arbitrary test function $\vec N \in \mathbf{X}$, 
integrating over $\Omega$ and imposing (\ref{eq:genblochflowb}) weakly, we obtain 
\begin{align}\label{eq:weakformulation}
&(\partial_t \vec M, \vec N)_{\mathbf{H}} + 
((\mathbf u \cdot \boldsymbol\nabla)\mathbf{M},\vec N)_{\mathbf{H}} + 
\gamma (\vec B\times \vec M,\vec N)_{\mathbf{H}}  \\ 
\nonumber
&\quad + (D \vec M,\vec N)_{\mathbf{H}}  + 
\int_{\Gamma} (\vec u \cdot \vec n)^{\ominus} \vec M \cdot \vec N \; \dif{s}  =    
(\vec f,\vec N)_{\mathbf{H}}  +
\int_{\Gamma} (\vec u \cdot \vec n)^{\ominus} \vec M_\Gamma \cdot \vec N \; \dif{s},
\end{align}
where $ w^\ominus(\vec r):= \frac{1}{2} (\abs{w(\vec r)} - w(\vec r))$ and 
$\vec M \times \vec B = - \vec B \times \vec M$.
Define bilinear and linear forms 
\begin{subequations}
	\begin{align}
	a(t;\vec M,\vec N)&:= ((\mathbf u \cdot \boldsymbol\nabla)\mathbf{M},\vec N)_{\mathbf{H}}   
	+ \gamma ((\vec B\times \vec M ),\vec N )_{\mathbf{H}} \nonumber \\ 
	&  + (D \vec M, \vec N)_{\mathbf{H}} + 
	\int_{\Gamma} (\vec u \cdot \vec n)^{\ominus} \vec M \cdot \vec N \dif{s}, \\
	l(\vec N)&:= (\vec f, \vec N)_{\mathbf{H}}   +
	\int_{\Gamma} (\vec u \cdot \vec n)^{\ominus} \vec M_\Gamma \cdot \vec N \dif{s},
	\end{align}
\end{subequations}
Then we obtain the variational form of the Bloch problem: Find $\vec M: (0,T] \to X$ s.t.
\begin{align}\label[problem]{prb:abslin}
(\partial_t \vec M, \vec N)_{\mathbf{H}}+a(t;\vec M,\vec N) & =l(\vec N), \quad 
\forall \; \vec N \in \mathbf{X}, \\
%   \end{align}x
%   \begin{align}
\label[problem]{prb:abslininit}
\eval[0]{\vec M}_{t=0}& =\vec M^0.
\end{align}
This is a Friedrichs system, see \cite{ern2004theory}, Sec. 7. 
Unfortunately, the theory in \cite{ern2004theory, Di2011mathematical} is not applicable since some 
coefficients are time-dependent. 
For $0 < \epsilon \ll 1$, we set $\mathbf{X}_\epsilon := [W^{1,2}(\Omega)]^3$ and consider an 
elliptic regularization of (\ref{prb:abslin}):  Find $\vec M_\epsilon: (0,T] \to X_\epsilon$ s.t.
\begin{align}\label[problem]{prb:epslin}
(\partial_t \vec M_\epsilon, \vec N)_{\mathbf{H}} +a_\epsilon(t;\vec M_\epsilon,\vec N) & =l(\vec N),\quad\forall \; 
\vec N \in \mathbf{X_\epsilon}, \\
\label[problem]{prb:epslininit}
\eval[0]{\vec M_\epsilon }_{t=0}& =\vec M^0.
\end{align}
with 
\begin{align}\label[problem]{prb:epslinA}
a_{\epsilon}(t;\vec M,\vec N) &:= a(t; \vec M, \vec N) + \epsilon (\nabla {\vec M},\nabla {\vec N})_{\mathbf{H}}.
\end{align}
Please note that the variational formulation of the regularized Bloch problem incorporates 
do-nothing boundary conditions $\epsilon \nabla \vec M_\epsilon \cdot \mathbf{n}= \vec 0$ on 
$\Gamma_0 \cup \Gamma_+$.

The spaces $\mathbf{X_\epsilon} \subseteq \mathbf{H}$ and the dual space $\mathbf{X_\epsilon}^*$ form 
an evolution triple 
$(\mathbf{X_\epsilon},\mathbf{H},\mathbf{X_\epsilon}^*)$.
For $p \ge 1$ and Banach space $\mathbf{Y}$ we denote by $L^p(0,T;\mathbf{Y})$ the Bochner space of 
vector-valued functions $\vec v: (0,T) \to \mathbf{Y}$. We look for a solution 
$\vec M_\epsilon \in L^\infty(0,T;\mathbf{H}) \cap L^2(0,T;\mathbf{X_\epsilon})$ of problem 
(\ref{prb:epslin})-(\ref{prb:epslininit}). 
%A (potentially existing) limit $M:= \lim_{\epsilon \to +0} M_\epsilon$ 
%will be seen as weak solution of (\ref{prb:abslin}). 

\begin{theorem}[Well-posedness]\label{thm:genexistunique}
	For all $\epsilon >0$, for given $\vec u \in [L^\infty (0,T;W^{1,\infty}(\Omega)]^3$ with 
	$\mbox{div}~{\vec u}=0$ and $\vec B \in [L^\infty(0,T;\mathbf{H})]^3$, there exists a unique solution 
	$\vec M_\epsilon \in L^\infty(0,T;\mathbf{H}) 
	\cap L^2(0,T;\mathbf{X_\epsilon})$ to \Crefrange{prb:epslin}{prb:epslininit}. 
	For $t \in (0,T]$ and with $\sigma := \frac{1}{T_1}$, the following 
	a-priori estimate is valid
	\begin{align} \nonumber
	& \frac{1}{2}\norm{\vec M_\epsilon(t)}_{\mathbf{H}}^2 + \int_0^t \e^{\sigma(\tau -t)} \left[ \epsilon \norm{\nabla 
		\vec M_\epsilon(\tau)}^2_{\mathbf{H}} 
	+ \frac12\int_{\Gamma}|(\vec u \cdot \vec n)| (\vec M_\epsilon \cdot \vec M_\epsilon)(s,\tau)\dif s \right] \dif{\tau}
	\\ & \leq \frac{1}{2}\norm{\vec M_\epsilon(0)}_{\mathbf{H}}^2 \e^{-\sigma t} 
	+ \frac{1}{2\sigma} \int_0^t \norm{\vec f(\tau)}_{\mathbf{H}}^2 \e^{\sigma(\tau -t)} \; \dif{\tau}.
	\label{eps-apriori}
	\end{align}
%{ For constant source term $\vec f$, we obtain 
%	\[
%	\frac{1}{2\sigma}  \int_0^t \norm{\vec f(\tau)}_{\mathbf{H}}^2 \e^{\sigma(\tau-t)} \dif{\tau} = 
%	\frac{1}{2\sigma}   (\frac{M_0}{T_1})^2 \abs{\Omega} \frac{1-\e^{-\sigma t}}{\sigma}.
%	\]
%}
\end{theorem}
\begin{proof}
	In order to prove well-posedness, we apply the main existence theorem for evolution problems by J.L. Lions, 
	see Theorem 6.6 in \cite{ern2004theory}. To this end, we define the norm
	\begin{align}
	\label{eq:eps-norm}
	\norm{\vec M}_{\mathbf{X}_\epsilon} = \left[ \epsilon \norm{\nabla {\vec M}}^2_{L^2(\Omega)} + 
	\norm{\vec M}^2_{\mathbf{X}}  \right]^{\frac12}.
	\end{align}
	For the application of the theorem, the following conditions must be satisfied: 
	\begin{enumerate}[label=({\bfseries P\arabic*})]
		\item The time-dependent bilinear form $t \mapsto a_\epsilon (t;\vec M,\vec N)$ is measurable 
		for all $\vec M,\vec N \in \mathbf{X}_\epsilon$ since the fields $\vec u $ and $\vec B$ are 
		sufficiently smooth.
		\item Application of the Cauchy-Schwarz and generalized H\"older inequalities show that the bilinear 
		form $a_\epsilon(t;\cdot,\cdot)$ is bounded for 
		$t \in [0,T]$ and for all $\vec M, \vec N \in \mathbf{X}_\epsilon$:
		\begin{align}
		\abs{a_\epsilon (t;\vec M, \vec N)}\leq & \epsilon \norm{\nabla {\vec M}}_{\mathbf{H}} 
		\norm{\nabla {\vec N}}_{\mathbf{H}} + \abs{(\vec u \cdot \nabla\vec M,\vec N)_{\mathbf{H}} } + 
		\gamma \abs{(B \times \vec M, \vec N)_{\mathbf{H}} }  \nonumber \\ & 
		+\abs{(D\vec M, \vec N)_{\mathbf{H}} } + 
		\abs{\int_{\Gamma} (\vec u \cdot \vec n)^{\ominus} \vec M \cdot \vec N \; \dif{s}} 
		\nonumber\\ \nonumber
		\leq & \epsilon \norm{\nabla {\vec M}}_{\mathbf{H}} \norm{\nabla {\vec N}}_{\mathbf{H}} + 
		\norm{\vec u \cdot \nabla \vec M}_{\mathbf{H}} \norm{\vec N}_{\mathbf{H}} + 
		\gamma \norm{\vec B}_{L^\infty}\norm{\vec M}_{\mathbf{H}} \norm{\vec N}_{\mathbf{H}} \nonumber \\
		&+ \norm{D}_{L^\infty} \norm{\vec M}_{\mathbf{H}} \norm{\vec N}_{\mathbf{H}}+
		k_s \norm{\vec M}_{\mathbf{X}} \norm{\vec N}_{\mathbf{X}} 
		\nonumber   
		\end{align}  
		with $\norm{\vec B}_\infty := \norm{\vec B}_{L^\infty(0,T;[L^\infty (\Omega)]^3)}$, 
		$\norm{D}_{L^\infty} := \max \lbrace \frac{1}{T_1},\frac{1}{T_2} \rbrace = \frac{1}{T_2}$ as $T_1 \geq T_2$ 
		and where we used the trace 
		inequality (see \cite{Di2011mathematical}, Lemma \num{2.5})
		\begin{align}
		\abs{\int_{\Gamma} (\vec u \cdot \vec n)^{\ominus} \vec M \cdot \vec N \; \dif{s}} \leq 
		k_s \norm{\vec u \cdot \nabla \vec M}_{\mathbf{H}} \norm{\vec u \cdot \nabla \vec N}_{\mathbf{H}} 
		\le  k_s \norm{\vec M}_{\mathbf{X}} \norm{\vec N}_{\mathbf{X}} . \nonumber
		\end{align}
		Now the norm definitions (\ref{eq:graphnorm}) and (\ref{eq:eps-norm}) imply boundedness:
		{
		\begin{align}
		\abs{a_\epsilon(t;\vec M, \vec N)}\leq & (2+  % 1+ \epsilon
		 \gamma \norm{\vec B}_{L^\infty} + 
		\norm{D}_{L^\infty} + k_s) \norm{\vec M}_{\mathbf{X_\epsilon}} \norm{\vec N}_{\mathbf{X_\epsilon}}. \nonumber
		\end{align}
	}
		\item Finally, bilinear form $a_\epsilon$ fulfills a coercivity condition. Integration by parts yields
		\begin{align}\nonumber
		a_\epsilon(t;\vec N,\vec N) &= \epsilon \norm{\nabla \vec N}^2_{\mathbf{H}} + 
		\norm{D^{\nicefrac{1}{2}}\vec N}^2_{\mathbf{H}}  + 
		\frac{1}{2}\int_{\Gamma} (\vec u \cdot \vec n) \vec N \cdot \vec N \; \dif{s} + 
		\int_{\Gamma} (\vec u \cdot \vec n)^{\ominus} \vec N \cdot \vec N \; \dif{s} \\ \nonumber
		&= \epsilon \norm{\nabla \vec N}^2_{\mathbf{H}} + \norm{D^{\nicefrac{1}{2}}\vec N}_{\mathbf{H}}^2 +
		\frac{1}{2}\int_{\Gamma} \abs{(\vec u \cdot \vec n)} \vec N \cdot \vec N \; \dif{s} \\
		& \geq \epsilon \norm{\nabla \vec N}^2_{\mathbf{H}} + \sigma \norm{\vec N}_{\mathbf{H}}^2 + 
		\frac{1}{2}\int_{\Gamma} \abs{(\vec u \cdot \vec n)} \vec N \cdot \vec N \; \dif{s}
		\label{eq:coercivity}
		\end{align}
		with
		\begin{align}
		D^{\nicefrac{1}{2}}&:= \diag{\frac{1}{\sqrt{T_2}},\frac{1}{\sqrt{T_2}},\frac{1}{\sqrt{T_1}}}, \quad 
		\sigma:= \min(\frac{1}{T_1},\frac{1}{T_2}) = \frac{1}{T_1}\; (\text{as}\;T_1\geq T_2),
		\end{align}
		as $(\vec B \times \vec N, \vec N )= 0$ and $\mbox{div}~{\vec u}=0$ due to the incompressibility assumption. 
	\end{enumerate}
	Now we can apply the theorem by Lions giving existence and uniqueness of a generalized solution 
	$\vec M_\epsilon:[0,T]\mapsto \mathbf{X}$ of (\ref{prb:epslin})--(\ref{prb:epslininit}). %\labelcref{eq:genblochflow}. 
	
	It remains to prove the a-priori estimate. 
	For all $t\in [0,T]$, we take $\vec N= \vec M_\epsilon$ in (\ref{prb:epslin})--(\ref{prb:epslininit}) 
	%\Crefrange{prb:epslin}{prb:epslininit}
	%\labelcref{eq:weakformulation} 
	to obtain via Cauchy-Schwarz and Young inequalities 
	\begin{align} \nonumber
	\frac{1}{2}\dod{}{t}\norm{\vec M_\epsilon}^2_{\mathbf{H}} + a_\epsilon(t;\vec M_\epsilon,\vec M_\epsilon) 
	&= (\vec f,\vec M_\epsilon)_{\mathbf{H}}
	\leq \frac{1}{2\sigma}\norm{\vec f}_{\mathbf{H}}^2 + \frac{\sigma}{2}\norm{\vec M_\epsilon}_{\mathbf{H}}^2, 
	\end{align}
	then via ({\bf{P3}}) 
	\begin{align} \nonumber
	\frac{1}{2} \dod{}{t} \norm{\vec M_\epsilon}_{\mathbf{H}}^2 + \epsilon \norm{\nabla \vec M_\epsilon}^2_{\mathbf{H}} 
	+ \frac{\sigma}{2} \norm{\vec M_\epsilon}_{\mathbf{H}}^2
	+\frac{1}{2}\int_{\Gamma} \abs{(\vec u \cdot \vec n)} \vec M_\epsilon \cdot \vec M_\epsilon \; \dif{s} 
	\leq \frac{1}{2\sigma} \norm{\vec f}_{\mathbf{H}}^2.
	\end{align}
	Now, application of the Gronwall Lemma, see Lemma \num{6.9} in \cite{ern2004theory}, implies
	\begin{align} \nonumber
	& \frac{1}{2}\norm{\vec M_\epsilon}_{\mathbf{H}}^2 + 
	\int_0^t \e^{\sigma(\tau -t)} \left[ \epsilon \norm{\nabla \vec M_\epsilon(\tau)}^2_{\mathbf{H}} 
	+ \frac12\int_{\Gamma}|(\vec u \cdot \vec n)| (\vec M_\epsilon \cdot \vec M_\epsilon)(s,\tau) \dif s \right] \diff{\tau}
	\\ & \leq \frac{1}{2}\norm{\vec M_\epsilon(0)}_{\mathbf{H}}^2 \e^{-\sigma t} 
	+ \frac{1}{2\sigma} \int_0^t \norm{\vec f(\tau)}_{\mathbf{H}}^2 \e^{\sigma(\tau -t)} \; \dif{\tau}.
	\end{align}
	{
	and we obtain estimate (\ref{eps-apriori})}.
\end{proof}
{
	\begin{remark}
	For the given case of constant source $\vec f$, we obtain 
	\begin{align}
	\norm{\vec f(t)}_{\mathbf{H}}^2=\norm{\vec f}_{\mathbf{H}}^2 = \int_\Omega (\frac{M_0}{T_1})^2 \dif{\vec r}
	= (\frac{M_0}{T_1})^2 \abs{\Omega}, \qquad \abs{\Omega} = \int_\Omega \dif{\vec r}
	\end{align}
	hence
	\begin{align} \nonumber
	\int_0^t \norm{\vec f(\tau)}_{\mathbf{H}}^2 \e^{\sigma(\tau-t)} \dif{\tau} = 
	(\frac{M_0}{T_1})^2 \abs{\Omega} \frac{1-\e^{-\sigma t}}{\sigma}.
	\end{align}
	%Finally, with $\sigma = \frac{1}{T_1}$, we obtain estimate (\ref{eps-apriori}).
	%\begin{align} \label{eq:aprioriestimate}
	%\\frac{1}{2}\norm{\vec M (t)}_{\mathbf{H}}^2 \leq & \frac{1}{2}\norm{\vec M(0)}_{\mathbf{H}}^2 \e^{-\frac{t}{T_1}}+ 
	%\          \frac{M_0^2}{2} (1-\e^{-\frac{t}{T_1}}) \abs{\Omega}.
	%\\end{align}

%
\end{remark}
}
Finally, we can pass to the limit { $\epsilon \to +0 $}, i.e. to the Bloch model. 
\begin{theorem}\label{thm:energyestimate}
	The Bloch model (\ref{prb:abslin}) admits a unique solution $\vec M \in L^\infty(0,T;\mathbf{H}) 
	\cap L^2(0,T;\mathbf{X})$. The kinetic energy of the magnetic field is bounded by:
	\begin{align}  \nonumber 
	&  \frac{1}{2}\norm{\vec M(t)}_{\mathbf{H}}^2 + \frac12 \int_0^t \e^{\sigma(\tau -t)} \left[  
	\int_{\Gamma}|(\vec u \cdot \vec n)| (\vec M \cdot \vec M))(s,\tau) \dif s \right] \diff{\tau}
	\\ 
	& \leq \frac{1}{2}\norm{\vec M(0)}_{\mathbf{H}}^2 \e^{-\sigma t} 
	+ \frac{1}{2\sigma} \int_0^t \norm{\vec f(\tau)}_{\mathbf{H}}^2 \e^{\sigma(\tau -t)} \; \dif{\tau}.
	\label{eq:energyestimate}
	\end{align}
\end{theorem}
\begin{proof}
	A careful inspection of the proof of \Cref{thm:genexistunique} shows that the existence and 
	uniqueness result together with the a-priori estimate remain valid for $\epsilon \to +0$. 
	As already mentioned, the variational formulation of the regularized Bloch problem incorporates 
	do-nothing boundary conditions $\epsilon \nabla \vec M_\epsilon \cdot \mathbf{n}= \vec 0$ on 
	$\Gamma_0 \cup \Gamma_+$. In the proof of the limit $\epsilon \to +0$ at $\Gamma_0 \cup \Gamma_+$, 
	one can proceed as in Chapter V.1 of \cite{lions2006perturbations}. Note that here it is used that 
	inflow and outflow are separated.
\end{proof}
\begin{remark}
	The result of \Cref{thm:genexistunique} and a-priori estimate ~\ref{eq:energyestimate} remain 
	valid for the special case $\vec u = \vec 0$, i.e. Bloch equations for spatially stationary objects.
\end{remark}

\section{Semi-discrete Equation} \label{sec:semidisc}
Here  we consider the spatial discretization of problem (\ref{prb:abslin}) by dG-FEM 
\cite{Di2011mathematical}. 
\subsection{Discontinuous Galerkin Formulation}
Consider a non-overlapping decomposition $\mathcal{T}_h := \{\Omega_i\}_{i=1}^I$ 
into convex simplicial subdomains $\Omega_i, i=1,2,\cdots , I$ as depicted in the left part of 
\Cref{fig:grid2d}.
We define the discontinuous finite element space
\begin{align}
[\mathbb{P}_k(\mathcal{T}_h)]^d := \{\vec N_h\in \mathbf{H};\: 
\eval[0]{\vec N_h}_{\Omega_i} \in [\mathbb{P}_k(\Omega_i)]^d \quad \forall \; \Omega_i,\; i=1,2, \cdots, I\}
\end{align}
where $\mathbb{P}_k$ denotes the set of polynomials of degree $k\in \mathbb{N}$. 
Moreover, let $\mathbf{X}_h = [\mathbb{P}_k(\mathcal{T}_h)]^d \bigcap \mathbf{X}$.
For adjacent subdomains $\Omega_i, \Omega_j$ with interface 
$E=\Gamma_{ij} = \overline{\Omega}_i \bigcap \overline{\Omega}_j$ and 
unit normal vector $\vec n_{ij}$ (directed from $\Omega_i$ to $\Omega_j$), 
as depicted in the right part of \Cref{fig:grid2d}, we define the average and jump of 
$\vec N_h \in \mathbf{X}$ across $\Gamma_{ij}$ by
\begin{subequations}
	\begin{align}\label{eq:dgavg}
	\left\langle \vec N_h \right \rangle_{\Gamma_{ij}}(\vec r)&:=\frac{1}{2}(\eval[0]{\vec N_h}_{\Omega_i}(\vec r) 
	+ \eval[0]{\vec N_h}_{\Omega_j}(\vec r)),\\
	[\vec N_h]_{\Gamma_{ij}}(\vec r)&:= \eval[0]{\vec N_h}_{\Omega_i}(\vec r) - 
	\eval[0]{\vec N_h}_{\Omega_j}(\vec r).\label{eq:dgjump}
	\end{align}
\end{subequations}

\begin{figure}[t]
	\centering
	\includegraphics[width=0.90\columnwidth]{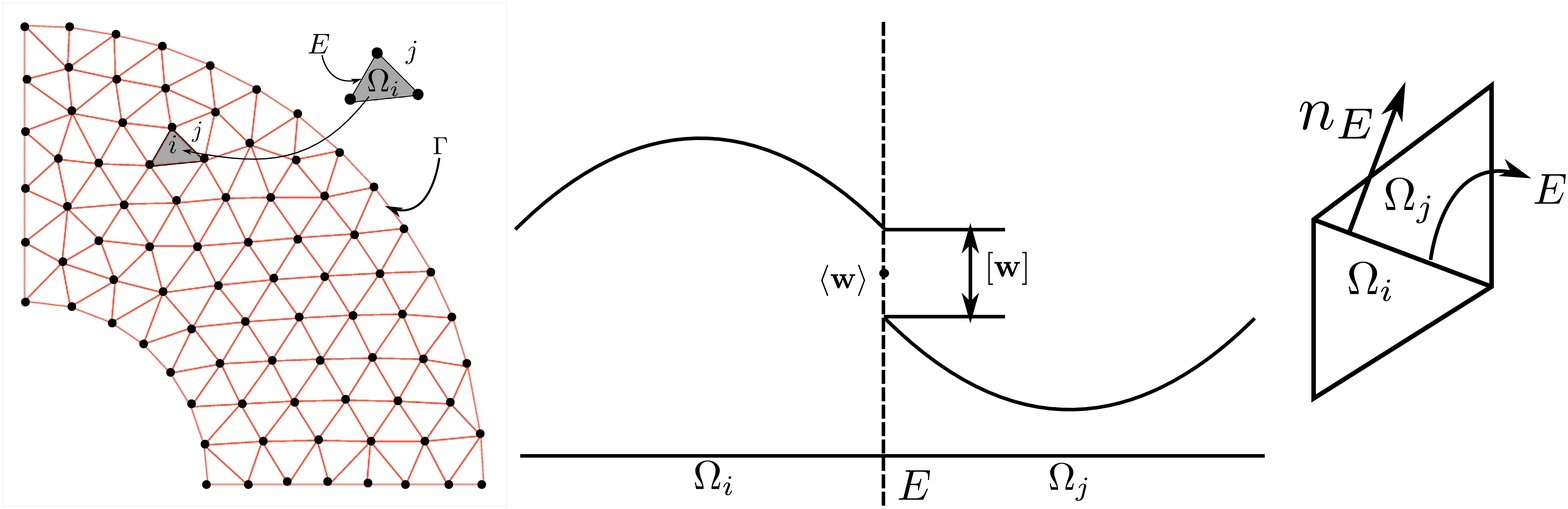}
	% \def\svgwidth{0.9\columnwidth}
	%\includesvg[svgpath=Chapter7/Figs/Vector/]{MRPhenomena}
	% CartesianSamplingFlash}
	%\input{MRPhenomena.pdf_tex}
	% \import{Images/}{traingulation.pdf_tex}
	\caption{Left: 2{D}-simplicial mesh. Right: 1D-example of average and jump operators. 
		Interface $E$ between $i$-th and $j$-th cell ($j>i$) is depicted with the used notation. 
		Orientation of outward normal is from lower to higher numbered cell.}
	\label{fig:grid2d}
\end{figure}

Let $\mathcal{F}_h^i$ be the set of all the interior interfaces $ E \subseteq \Omega$ and define 
the upwind form % (with $\eta \ge 0$)
\begin{align}
S_h(t;\vec M,\vec N):=& 
%  \int_\Gamma ( \vec u \cdot \vec n)^\ominus (\vec M \cdot \vec N)\; \dif{s} - 
\sum_{E\in \mathcal{F}_h^i}\int_{E} \bigg(-(\vec u \cdot \vec n_E) [ \vec M] \cdot \left\langle \vec N \right \rangle 
+\frac{1}{2} \abs{\vec u \cdot \vec n_E} [\vec M]\cdot [\vec N]\bigg) \; \dif{s} .
\end{align}
Moreover, gradient jumps over interior faces are penalized via
\begin{align} \label{eq:penalize}
p_\epsilon(\vec M,\vec N):= {\tilde \epsilon} \sum_{E \in \mathcal{F}_h^i} h_E^2  \int_E |\mathbf{u} \cdot \mathbf{n}_E| 
[\nabla \vec M]_E : [\nabla \vec N]_E\; \dif{s}, \quad {\tilde \epsilon} \ge 0. 
\end{align}
Setting
\begin{align} \label{eq:upwind}
a_\epsilon ^{\textnormal{upw}}(t;\vec M, \vec N):= p_\epsilon(\vec M,\vec N) 
+ a(t;\vec M, \vec N) + S_h(t;\vec M, \vec N),
\end{align}
the upwind dG-FEM reads: find $ \vec M_h:(0,T] \mapsto \mathbf{X}_h$ such that $\forall\; \vec N_h \in \mathbf{X}_h$
\begin{align}\label{eq:discretebilinear}
(\partial_t\vec M_h, \vec N_h)_{\mathbf{H}} +   a_\epsilon^{\textnormal{upw}}(t;\vec M_h, \vec N_h) = l(\vec N_h).
\end{align}

\subsection{Well-posedness of the Semi-discrete Problem}
Let us define the  norm $\opnorm{\vec N_h}_{\mathbf{U}}$ via
\begin{align} \nonumber
\opnorm{\vec N_h}_{\mathbf{U}}^2&:= % \sigma \norm{\vec N_h}_{L^2(\Omega)}^2 + 
{\tilde \epsilon} \sum_{E \in \mathcal{F}_h^i} h_E^2 \int_E |\mathbf{u} \cdot \mathbf{n}_E| 
\norm{[\nabla \vec N_h]_E }^2_{L^2(E)} \\
& + \frac{1}{2}\int_\Gamma \abs{\vec u \cdot \vec n}\abs{\vec N_h}^2 \dif{s}
+ \frac{1}{2}\sum_{E \in \mathcal{F}_h^i} \int_E \abs{\vec u \cdot \vec n_E} [\vec N_h]_E^2 \dif{s}.
\label[expression]{eq:discretetestfun}
\end{align}

\begin{theorem}
	The semi-discrete problem ~\ref{eq:discretebilinear} is well-posed and admits the a-priori estimate
	\begin{align}   \label{eq:semiwell} 
	\frac{1}{2}\norm{\vec M_h(t)}_{\mathbf{H}}^2  + 
	\int_0^t e^{\sigma (\tau -t)} \opnorm{\vec M_h (\tau)}_{\mathbf{U}}^2 \dif{\tau}  
	\leq  \frac{1}{2}\norm{\vec M_h(0)}_{\mathbf{H}}^2 e^{- \sigma t}
	+ \frac{1}{2\sigma} \int_0^t e^{\sigma (\tau -t)} \norm{\vec f(\tau)}_{\mathbf{H}}^2 \dif{\tau} .
	\end{align}
\end{theorem}
\begin{proof}
	The existence and uniqueness proof follows the lines of the proof of \Cref{thm:genexistunique}.
	Similiar to the approach in \Cref{sec:Well-posedness}, symmetric testing $\vec N_h = \vec M_h$ 
	provides 
	\begin{align} \label{eq:coercive}
	a_\epsilon^{\textnormal{upw}}((t;\vec M_h, \vec M_h) & \ge \sigma \norm{\vec M_h}_{\mathbf{H}}^2+ 
	\opnorm{\vec M_h}^2_{\mathbf{U}}.
	\end{align}
	Then, Young inequality shows 
	\begin{align} \nonumber
	\frac{1}{2}\dod{}{t}\norm{\vec M_h(t)}_{\mathbf{H}}^2 + \frac{\sigma}{2} \norm{\vec M_h}_{\mathbf{H}}^2  
	+ \opnorm{\vec M_h}_{\mathbf{U}}^2  \leq \frac{1}{2\sigma} \norm{f}_{\mathbf{H}}^2 . 
	\end{align}
	Now, similarly to \Cref{thm:energyestimate}, application of the Gronwall lemma yields 
	the a-priori estimate (\ref{eq:semiwell}). 
	This shows the well-posedness of the semi-discretized Bloch problem.
\end{proof}

\subsection{Semi-discrete Error Estimate}
For the error of the spatial discretization we obtain the following result.
\begin{theorem}
	The error of spatial discretization is given by:
	\begin{align} \nonumber 
	&\frac{1}{2}\norm{(\vec M-\vec M_h)(t))}_{\mathbf{H}}^2  
	+ \frac12 \int_0^t \e^{\sigma (\tau -t)} \opnorm{(\vec M-\vec M_h)(\tau)}_{\mathbf{U}}^2 \dif{\tau} 
	\leq   \frac{1}{2}\norm{(\vec M-\vec M_h)(0)}_{\mathbf{H}}^2 \e^{-\sigma t} \\ 
	\label{eq:errorestimate}
	& + %\inf_{\vec N_h \in \mathbf{X}_h} 
	\int_0^t \e^{\sigma (\tau -t)} \Big( 
	{\opnorm{(\vec M -\vec \pi_h \vec M)(\tau)}^2_{\mathbf{U},\flat}} + 
	\delta \norm{(\vec M - \vec \pi_h \vec M)(\tau)}_{\mathbf{H}} + p_\epsilon(\vec M,\vec M)(\tau) \Big) \dif{\tau}
	\end{align}
	with the $L^2$-orthogonal projection $\vec \pi_h \vec M$ of $\vec M$ onto $\mathbf{X}_h$ and 
	$\sigma = \frac{1}{T_1}$. The norm {$\opnorm{ \cdot }_{U,\flat}$} and constant $\delta$ will be 
	defined within the proof.
	For a sufficiently smooth solution $\vec M \in L^\infty (0,T;[W^{k+1,2}(\Omega)]^3)$, {the right-hand side
	term in (\ref{eq:errorestimate})} is of order $\bigO{h^{2k+1}}$ with 
	$h=\max_{i} \textnormal{diam}(\Omega_i)$.
\end{theorem}
\begin{proof}
	The error equation for the error $\vec M - \vec M_h$ is given by
	\begin{align} \nonumber % \label{eq:erroreqn1}
	(\partial_t (\vec M - \vec M_h),\vec N_h)_{\mathbf{H}} + a^{upw}_\epsilon (t;\vec M - \vec M_h,\vec N_h) = 
	p_\epsilon (\vec M, \vec N_h) 
	\quad \forall \; \vec N_h \in \mathbf{X_h},~ t \in (0,T]~ \mbox{a.e.}.  
	\end{align}  
	%  for all $\vec N_h \in \mathbf{X_h}$ and $t \in (0,T]$ a.e.
	Let $\pi_h \vec N$ be the $L^2$-orthogonal projection of $\vec N$ onto $\mathbf{X}_h$, i.e., 
	$ (\vec N-\pi_h \vec N, \vec w) _{\mathbf{H}} = 0 \quad \forall \; \vec w \in \mathbf{X}_h.$
	Then we split the error as 
	$$\vec M - \vec M_h = (\vec M - \vec \pi_h \vec M) + (\vec \pi_h \vec M - \vec M_h) 
	\equiv \vec I_h + \vec E_h$$
	and reformulate the error equation % (\ref{eq:erroreqn1}) 
	with $\vec N_h=\vec E_h$ as
	\begin{align} \label{eq:erroreqn2}
	(\partial_t \vec E_h,\vec E_h)_{\mathbf{H}} + a^{upw}_\epsilon (t;\vec E_h,\vec E_h) = 
	-  a^{upw}_\epsilon (t;\vec I_h,\vec E_h) + p_\epsilon (\vec M, \vec E_h), \quad t \in (0,T]~ \mbox{a.e.} 
	\end{align}  
	where we used that $(\partial_t \vec I_h,\vec E_h)_{\mathbf{H}}=0$ due to $L^2$-orthogonality.
	The left-hand side can be bounded from below as
	\begin{align} \label{eq:lowerbound}
	(\partial_t \vec E_h,\vec E_h)_{\mathbf{H}} + a^{upw}_\epsilon (t;\vec E_h,\vec E_h)
	\ge \frac12 \frac{d}{dt} \norm{\vec E_h}^2_{\mathbf{H}} + \sigma \norm{\vec E_h}^2_{\mathbf{H}} 
	+ \opnorm{\vec E_h}_{\mathbf{U}}^2.
	\end{align}
	For an estimate of the right-hand side we use 
	$
	p_\epsilon(\vec M,\vec E_h) \le p_\epsilon(\vec M,\vec M) \opnorm{\vec E_h}_{\mathbf{U}}
	$ 
	and
	\begin{align} \label{eq:upperbound}
	-  a^{upw}_\epsilon (t;\vec I_h,\vec E_h) % + p_\epsilon (\vec M, \vec E_h) \le 
	\le \Big(  {\opnorm{\vec I_h}_{\mathbf{U},\flat}} + 
	\delta \norm{\vec I_h}^2_{\mathbf{H}} % + p_\epsilon(\vec M,\vec M) 
	\Big) \opnorm{\vec E_h}_{\mathbf{U}}
	\end{align}
	with 
	\begin{align} \label{eq:opnorm*}
	{\opnorm{\vec I_h}^2_{\mathbf{U},\flat}}:= \zeta \opnorm{\vec I_h}^2_{\mathbf{U}} + 
	\sum_{T \in \mathcal{T}_h} \norm{\vec u}_{L^\infty (\partial T)} \norm{\vec I_h}^2_{L^2(\partial T)}
	%          + \delta \norm{\vec I_h}^2_{\mathbf{H}} .
	\end{align} 
	and $\delta := \gamma \norm{\vec B}_{L^\infty} + \norm{D}_{L^\infty}$ resp. 
	$\zeta := \max \lbrace 1; \norm{\vec u}_{L^\infty (0,T;W^{1,\infty} (\Omega))} \rbrace .$
	Estimate (\ref{eq:upperbound}) relies on a generalization of Lemma 2.30 in \cite{Di2011mathematical} 
	for scalar advection-reaction problems to the vector-valued case. The estimate provides a 
	careful bound of the uwpind-discretized convective term and heavily exploits the boundedness of
	$L^2$-orthogonality of subscales.
	
	Combining (\ref{eq:lowerbound}), (\ref{eq:upperbound}) and using Young inequality, we obtain
	\begin{align} \label{eq:estimate}
	\frac12 \frac{d}{dt} \norm{\vec E_h}^2_{\mathbf{H}} + \frac12 \opnorm{\vec E_h}_{\mathbf{U}}^2 
	\le \opnorm{\vec I_h}^2_{\mathbf{U},\flat} + \delta \norm{\vec I_h}^2_{\mathbf{H}} + p_\epsilon(\vec M, \vec M).
	\end{align}
	Integration of (\ref{eq:estimate}) and the triangle inequality imply the quasi-optimal 
	error estimate (\ref{eq:errorestimate}).
	For a sufficiently smooth solution $\vec M \in L^\infty (0,T;[W^{k+1,2}(\Omega)]^3)$, the penalty term 
	$p_\epsilon(\vec M,\vec M)$ vanishes. Finally, interpolation results imply the error order 
	$\bigO{h^{k+\nicefrac{1}{2}}}$. 
\end{proof}

\begin{remark}
	The case $k=0$ covers the finite volume method (FVM).
\end{remark}

\section{Temporal Discretization} \label{sec:temporal}
Starting point is the spatially discretized problem: find $\vec M_h:(0,T]\mapsto \mathbf{X}_h
$ such that for all $ \mathbf{N}_h \in \mathbf{X}_h$
\begin{align}\label{eq:abstractdiscreteprob}
(\partial_t \mathbf{M}_h(t),\mathbf{N}_h)_{\mathbf{H}}+ a_\epsilon^{\textnormal{upw}}(t;\vec M_h(t),\vec N_h)=l(\vec N_h), 
  \quad \vec M_h(0) = \vec M_{h0}.
\end{align}
A major problem stems from the multiscale character of \labelcref{eq:abstractdiscreteprob} as the scale of 
magnetization is much faster than that of advection. Another difficulty is the restricted smoothness of the 
data in time, in particular of field $\vec G=\begin{pmatrix} G_x, G_y, G_z \end{pmatrix}^T$, see 
\Cref{fig:RadialFLASH} (for the FLASH sequence \cite{Frahm1985verfahren,haase1986flash}). 

According to the required high resolution in time, an explicit time stepping is chosen.
We considered two variants
\begin{enumerate*}[label=(\roman*)]
	\item a fully coupled approach and
	\item an operator splitting approach.
\end{enumerate*}
The efficiency of the numerical simulation can be strongly improved using GPU computing. This will
be exemplarily shown in \Cref{sec:experiments}.

\subsection{Fully Coupled Approach}\label{subsec:fullycoupled}
Following Sec. (\num{3.1}) in \cite{Di2011mathematical}, we apply a low-order explicit Runge-Kutta scheme to 
\labelcref{eq:abstractdiscreteprob}. Define a discrete operator
$A_\epsilon^{\textnormal{upw}}:\mathbf{X} + \mathbf{X}_h \mapsto \mathbf{X}_h$ via 
$ (A_\epsilon^{\textnormal{upw}}(t)\vec v, \vec w)_{\mathbf{H}}:=a_\epsilon^{\textnormal{upw}}(t;\vec v,\vec w)$. 
Similarly, let $L$ be a functional on $\mathbf{X}_h$ with $L=l(\vec w)$. Note that $L$ is constant in time in 
this application.

Let $0=t^0< t^1<t^2<\cdots<t^N =T$ be the set of discrete times with time steps $\tau_n:= t^{n+1}-t^n, n=0,1,\cdots, N-1$. 
Moreover, we denote $\mathbf{M}_h^n=\mathbf{M}_h(t^n)$ etc.

A two-stage RK-scheme is a good compromise between temporal accuracy and the restricted data smoothness in time. 
Following Subsec. 3.1.3 in \cite{Di2011mathematical} we select scheme 
\begin{subequations}
	\begin{align}\label{eq:twostagerk}
	\vec{M}_h^{n,1} &=\vec M_h^n - \tau_n A_\epsilon^{\textnormal{upw}}
	\vec M_h^n+\tau_n L \\
	\vec{M}_h^{n+1} &= \frac{1}{2}(\vec M_h^n+\vec M_h^{n,1}) -\frac{1}{2}\tau_n A_\epsilon^{\textnormal{upw}}
	\vec M_h^{n,1}+\frac{1}{2}\tau_n L .
	\end{align}
\end{subequations}
Precise statements of the stability and convergence of \labelcref{eq:twostagerk} can be found in case of smooth 
data in Subsec. 3.1.6 of \cite{Di2011mathematical}. In particular, a time step restriction on $\tau_n$ comes from
a CFL condition for the advective term. As in this application the time step restriction on $\tau_n$ comes from 
magnetization, the mentioned CFL condition is always valid in our calculations.

We will not repeat the details, e.g. of the stability and convergence RK2-analysis in \cite{Di2011mathematical}. 
It provides in case of smooth data in time, an error of order $\bigO{\tau_n^2+h^{k+\nicefrac{1}{2}}}$ with polynomial 
degree $k$ of spatial discretization (see Theorem 3.10). Such error estimate in time is not valid in this 
application, since the data $\vec{B}(t,\vec r)$ are only in $[C^{0,1}[0,T]]^3$ (for the FLASH-sequence
studied in \Cref{sec:experiments}) or even only in $[L^\infty(0,T)]^3$ for the example in Subsec. 5.2.
Alternatively, a standard embedded RK-scheme of type RK3(2) or even of higher order like RK5(4) is chosen for 
appropriate time step selection, see \Cref{sec:experiments}.
For details of the embedded methods, we refer, e.g., to Sec. 5 of \cite{deuflhard2012scientific}.

\subsection{Operator Splitting}\label{subsec:splitting}
The above mentioned multiscale character of problem \labelcref{eq:abstractdiscreteprob} suggests an operator 
splitting approach as suggested e.g.in Sec. IV.1.5. of \cite{Hundsdorfer2003} for large advection-reaction 
problems, e.g. in air-pollution simulations. Writing the semi-discrete problem \labelcref{eq:abstractdiscreteprob} 
with
\begin{align}
 \vec F_{\textnormal{adv}}(t,\vec M):= -\nabla\cdot(\vec u \vec M),\quad 
 \vec F_{\textnormal{mag}}:= \gamma \vec B \times \vec M + D\vec M -\vec f,
\end{align}
formally as ODE-system
\begin{align}
\dod{\vec M_h(t)}{t}=\vec F_{\textnormal{adv}}(t,\vec M_h(t))+ \vec F_{\textnormal{mag}}(t,\vec M_h(t)),
\end{align}
the simplest sequential operator splitting on $t^n\leq t \leq t^{n+1}$ gives
\begin{subequations}\label{eq:operatorsplitting}
\begin{align}
\dod{\vec M_h^*(t)}{t}&=\vec F_{\textnormal{adv}}(t,\vec M_h^*(t)),&&\vec M_h^*(t) = \vec M_h(t^n), \\
\dod{\vec M_h^{**}(t)}{t}&=\vec F_{\textnormal{mag}}(t,\vec M_h^{**}(t)),&&\vec M_h^{**}(t) = \vec M_h^*(t^n).
\end{align}	
\end{subequations} 
An inspection of \labelcref{eq:operatorsplitting} shows for the splitting error at $t=t^n$
\begin{align}
\bm\epsilon_s = \frac{1}{2}\tau_n^2[\dpd{\vec F_{\textnormal{adv}}}{\vec M}\vec F_{\textnormal{mag}} - 
    \dpd{\vec F_{\textnormal{mag}}}{\vec M}\vec F_{\textnormal{adv}}] +\bigO{\tau_n^3}.
\end{align}
It turns out that the commutation error $[\vec F_{\textnormal{adv}}, \vec F_{\textnormal{mag}}] =
        \dpd{\vec F_{\textnormal{adv}}}{\vec M_h}\vec F_{\textnormal{mag}} - \dpd{\vec F_{\textnormal{mag}}}{\vec M_h}\vec F_{\textnormal{adv}}$ 
can be written as
\begin{align}
[\vec F_{\textnormal{adv}}, \vec F_{\textnormal{mag}}]= (\nabla \cdot \vec u)[\vec F_{\textnormal{mag}}(t,\vec M_h)-
  \dpd{\vec F_{\textnormal{mag}}}{t}(t,\vec M_h)\vec M_h(t)] + (\vec u\cdot \nabla_{\vec r})\vec F_{\textnormal{mag}}(t,\vec M_h).
\end{align}
It vanishes if either $\vec F_{\textnormal{mag}}$ is independent of $\vec r = (x, y, z)^T $ and 
$\mbox{div}~\vec u = 0$ or  $\vec F_{\textnormal{mag}}$ is independent of $\vec r$ 
and linear in $\vec M_h$, see \cite{Hundsdorfer2003}, sec IV.1.5. This is unfortunately not the case in this 
application. As a remedy, symmetric or Strang-Marchuk splitting can be applied which reduces the splitting 
error to $\bigO{\tau_n^2}$, see \cite{Hundsdorfer2003}.

\section{Numerical Experiments} \label{sec:experiments}
The simulation method was tested for both static objects, i.e. with $\vec u \equiv \vec 0$,
as well as flow experiments, i.e. with $\vec u \not\equiv \vec 0$. 

\subsection{Experimental Validation for Static Objects ($\vec u \equiv \vec 0$)} \label{subsec:5.1}

\begin{figure}[t]
  \centering
    \includegraphics[width=0.80\linewidth]{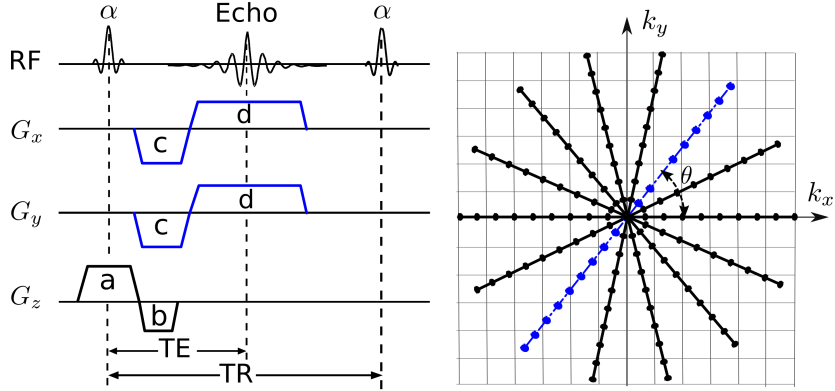}
  \caption{Generic spoiled gradient echo sequence with radial trajectory. Gradients: (a) slice selection 
      (b) rewinder (c) prephasing (d) readout. The dashed radial line (spoke) in the $\vec k$-space corresponds to the current repetition.} \label{fig:RadialFLASH}
  \end{figure}
  
We start with a test of the simulator for the static case, i.e. for $\vec u = \vec 0$. This means that 
mainly the temporal discretization is considered. The experiments here and later on (with exception of 
Subsec.~\ref{subsec:5.2}) were performed using a randomly spoiled \cite{Roeloffs2015spoiling} FLASH sequence 
\cite{Frahm1985verfahren,haase1986flash} with a radial trajectory as shown in \Cref{fig:RadialFLASH}. 
The left part depicts the time diagram of gradients and the right part shows the $\vec k$-space sampling 
trajectory where 
\[
 k_x :=\frac{\gamma}{2\pi}\int_0^tG_x(\tau)\dif{\tau}, \qquad  
 k_y:=\frac{\gamma}{2\pi}\int_0^tG_y(\tau)\dif{\tau}
\]
with $G_x = G_{\textnormal{max}}\cos(\theta)$, $G_y = G_{\textnormal{max}}\sin(\theta)$, 
$G_{\textnormal{max}} = \sqrt{G_x^2+G_y^2}$ and  angle $\theta$ of the spoke with $x$-axis.
%
%Replacing the $\vec k$-space expressions in (\ref{eq:recsig}) shows the relation between the acquired signal 
%during the readout in each repetition and the corresponding line in the $\vec k$-space.
	%\begin{wrapfigure}{r}{0.40\textwidth}
	%\end{wrapfigure}
Putting the expression of $\vec k$ in the demodulated signal equation, we can obtain the Fourier pair relation between the acquired electric signal and the transverse magnetization \cite{nishimura1996principles}. 
Each radial spoke in $\vec k$-space, which represents a projection of the object, is acquired with a repetition of pulse sequence where transverse gradients are changed according to orientation $\theta$ of the spoke in $\vec k$-space.
%For imaging with radial FLASH sequence, the pulse sequence is repeated multiple times and the transverse 
%gradients are changed to give different lines in $\vec k$-space which also means different projections 
%of the object.
Certain number of spokes in $\vec k$-space are required to reconstruct an image frame.

Carr showed in \cite{Carr1958} that under constant flip angle $\alpha$, and gradient moment and constant TR the magnetizations reach a state of dynamic equilibrium after several repetitions. For clinical imaging, the acquisition starts only after the magnetizations reach dynamic equilibrium after several preparatory TR repetition.
 In order to test the simulations, time-series data of images from the beginning of the preparatory phase to dynamic equilibrium were obtained from experiments and compared with equivalent simulation results.

\begin{wraptable}{r}{0.4\textwidth}
	\centering
	\caption{Doped water tubes with their relaxation times}
	\label{table:staticphantom}
	\vspace{-0.5em}
	\begin{tabular}{|c c c|}
		%{ |s|p{2cm}|p{2cm}|  }
		% \begin{tabular}{c c c }
		%\hline
		%\multicolumn{3}{|c|}{Doped Water Tubes}\\
		\hline
		\rowcolor{gray}Tube & $T_1$ [\si{\ms}] &$T_2$ [\si{\ms}]\\
		\hline
		\rowcolor{lightgray}
		3 & 296 & 113 \\
		%	\hline
		\rowcolor{lightgray}	4 & 463 & 53 \\
		%	\hline
		\rowcolor{lightgray}	7 & 604 & 95 \\
		%	\hline
		\rowcolor{lightgray}	10 & 745 & 157\\
		%	\hline
		\rowcolor{lightgray}	14 & 1034  & 167\\
		%	\hline
		\rowcolor{lightgray}	16 & 1276 & 204\\
		%	\hline
		\rowcolor{lightgray}	water & 2700 & 2100\\
		\hline
	\end{tabular}
\end{wraptable}
To this end, an experiment was performed with a phantom as can be seen in the left part of \Cref{fig:correctedimage}) 
containing multiple compartments of doped water tubes with known $T_1$ and $T_2$ as listed in 
\Cref{table:staticphantom}.
The relevant pulse sequence parameters in \Cref{fig:RadialFLASH} are: TR/TE = \SI{2.18/1.28}{\ms}, 
flip angle $\alpha$ =\ang{8}, and number of spokes per image frame= \num{27}. \num{100} image frames from the beginning to the dynamic equilibrium of the experiment were used for comparison with simulation results.

{ Before embarking on the validation of our code with experimental results, it was tested with a Bloch equation simulator written in MATLAB by Sun \cite{Sun2012bloch} where the computationally expensive parts are written in C and connected via mex interface \cite{Sun2012bloch}. The code is an extended and corrected version of the code by Hargreaves \cite{Hargreaves2004bloch}. The numerical results from our code agree really well with  \cite{Sun2012bloch}.
%	 In our method, we approximated the piece-wise constant gradient profile with continuous gradient profile and then used embedded RK45. 
	Baseline simulations were performed on a system with Supermicro SuperServer 4027GR-TR system with Ubuntu 14.04, 2x Intel Xeon Ivy Bridge E5-2650 main processors
	%Ubuntu 16.04, Intel Core(TM) i3-2330M, 2.20GHz and 2 GB RAM
	using an in-house C++ code and the implementation by Sun in Matlab R2016A. Though computationally RK5(4) solver is \num{4} times more expensive, the advantage of using RK5(4) is an easy higher order extension for flowing cases with any method of our choice for spatial discretization and we can avoid splitting error using a fully coupled approach.

Prior studies \cite{Shkarin1997time} have shown that, to model Bloch phenomena accurately, simulations must be carried out dividing each voxel with a number of cells where each of them represent one or an ensemble of isochromats; an isochromat is a microscopic group of spins which resonate at the same frequency. Shkarin et. al. showed in that \cite{Shkarin1997time} increasing the number of cells per voxel reduce the error of the solution.

For every simulation with relaxation time of specific liquids, the simulated data were recorded at TE since all isochromats are rephased at TE under the condition of approximately complete spoiling of residual transverse magnetization. 
The transverse magnetizations are integrated over all the isochromats and the magnitude of the summed up transverse magnetizations are averaged further for the number of spokes per image frame. In principle, this averaged integrated transverse magnetization should be equivalent to the averaged magnitude of image over a region in a specific tube. 
%
% to mimic intravoxel dephasing \cite{Shkarin1997time} populated with arbitrary number of isochromats with slightly varying off-resonance frequency; an isochromat is a microscopic group of spins which resonate at the same frequency.
%As observed from the simulation result, the number of subvoxels has an effect on mimicking the dynamics but isochromat ensemble within that subvoxel do not have much effect as illustrated in \ref{fig:subvoxels}
}

%\begin{figure}[t]
%	\centering
%	\includegraphics[width=1.0\linewidth]{subvoxels.eps}
%	\caption{Left: To validate illustrated in subsec. 3.4.3 of \cite{Hazra2016}
% 
%		Right: }
%	\label{fig:subvoxels}
%\end{figure}  
%for accuracy reasons, numerical MR simulations must be solved 
%for a large number of mutually independent isochromats ; 
Simulations were performed over a domain of \SI[product-units = power]{4.8 x 4.8 x 18.0}{\mm} 
(corresponding to $3 \times 3$ pixels in the $xy$-plane and three-times the nominal slice thickness=\SI{6}{\mm}
in $z$-direction for each tube) divided into \num{27 x 27 x 45} cells. Each cell is assumed to consist 
of one isochromat as shown in Subsec. 3.4.3. in \cite{Hazra2016} that the number of isochromats in each cell do not effect accuracy of the simulation. The initial condition is chosen as $\vec M^0 = (0, 0, 1 )^T$. The embedded RK5(4) scheme is applied for time discretization. The time series of these two equivalent quantities are plotted as a function of image frame number in \Cref{fig:correctedimage} (right) for four randomly chosen tubes after normalizing the experimental and simulated data by their respective magnitude in dynamic equilibrium of the brightest tube (Tube 3).

%with relative and absolute tolerance \num{1.0E-10} and \num{1.0E-10} respectively 
%
 
\begin{figure}[t]
	\centering
	\includegraphics[width=1.0\linewidth]{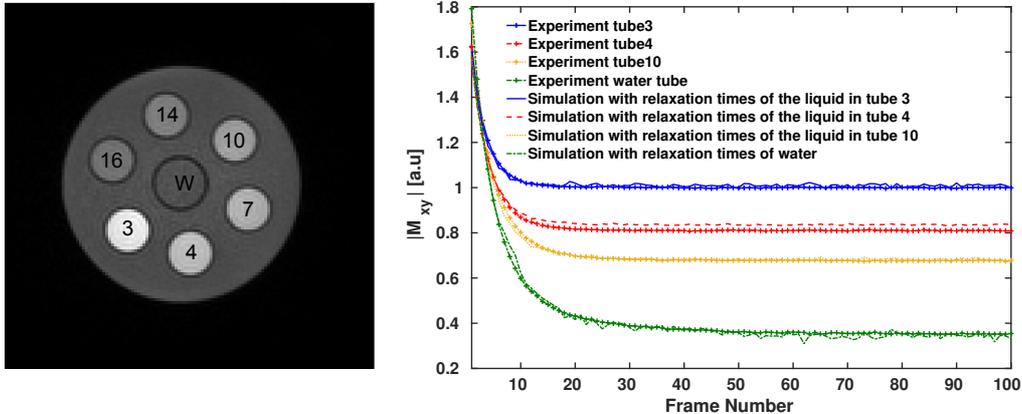}
	\caption{Left:  Image of the static phantom with resolution of \num{160x160} pixels. 
		Right: Comparison of simulation with the image for four different liquids.}
	\label{fig:correctedimage}
\end{figure}

The excellent agreement between simulation and experiment hints at the possible use of the simulator for quantitative estimations in MRI, e.g. the relaxation times $T_1$, $T_2$ or concentration of contrast agents required for certain signal enhancement (illustrated with an example in \cite{Hazra2016}, Sec. 5.4). 

Also, the time evolution of magnetizations for different isochromats are independent of
each other and suitable for Graphical Processing Unit (GPU) computing. A speed-up of 82 was achieved with GPU parallelization on the previously mentioned system together with a 
%a Supermicro SuperServer 4027GR-TR 
%system with Ubuntu 14.04, 2x Intel Xeon Ivy Bridge E5-2650 main processors and 
NVIDIA GTX Titan Black (Kepler GK110) GPU as illustrated in Subsec. 3.3.4 of \cite{Hazra2016}.

\subsection{A Basically One-dimensional Test Case for Flowing Spins} \label{subsec:5.2}
The simulator is tested further for flowing spins, i.e. $\mathbf{\vec u} \not\equiv \vec 0$. As first case, we consider a basically one-dimensional test case in \cite{Yuan1987solution} where they
studied for $z \in (-\frac{L}{2},\frac{L}{2})$ the effect of an RF pulse on the magnetization for 
different through-plane velocities, i.e. component $u_z$ in $z$-direction, using the
leap-frog finite difference scheme.

\begin{wrapfigure}{l}{0.40\textwidth}
	%\begin{figure}
	\centering
	\includegraphics[width=0.9\linewidth]
	{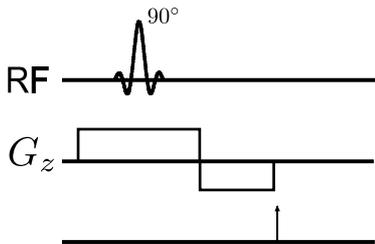}
	\caption{A \ang{90} slice-selective pulse was used for the studying the flow-effects. 
		The arrows indicate the time when the data was recorded.}
	\label{fig:yuanpulseseq}
	%\end{figure}
\end{wrapfigure}

As shown in \Cref{fig:yuanpulseseq}, a  Blackman-windowed sinc pulse with an amplitude of \SI{0.1750}{\gauss}, 
flip angle of \ang{90} and duration of \SI{2.6794}{\ms}, slice selection gradient 
$G_z = \SI{1.0}{\gauss \per \cm}$, and a nominal slice thickness of $\SI{7}{\mm}$ were used for the 
simulation in \cite{Yuan1987solution}. 
Through-plane velocities were in the range of \SIrange{0}{200}{\cm\per\s}. The simulations were performed 
over a length of $L=\SI{20}{\mm}$ and $L=\SI{30}{\mm}$, respectively, in the slice direction $z$, divided into 
\num{800} cells of size \SI{0.025}{\mm} and \SI{0.0375}{\mm} for the range of velocities 
\SIrange{0}{80}{\cm\per\s} and \SIrange{80}{200}{\cm\per\s} respectively. 
The time duration of simulations was divided into \num{4500} equidistant time steps of 
$\tau_n=\SI{8.9313e-04}{\ms}$.
The magnetizations were calculated at the end of the rewinder gradient marked by an arrow in 
\Cref{fig:yuanpulseseq} 

In this work, the simulations were performed with the FEM-package COMSOL Multiphysics using the dG-FEM with 
quadratic elements ($k=2$), see \Cref{sec:semidisc}. The %simulations were performed dividing the 
domain  of % \SI{20}{\mm} and 
L=\SI{30}{\mm} was divided into % \num{130} and 
\num{200} equidistant cells for the velocity range \SIrange{0}{200}{\cm\per\s}.
The fully coupled semi-discretized equations were further solved using the RK5(4) scheme.
For relative and absolute tolerance of \num{1.0E-6} and \num{1.0E-8} respectively, the adaptive time 
stepping resulted in \numrange{340}{528} time steps in the specified velocity range. The variability of 
time step sizes due to the embedded RK is depicted in \Cref{fig:timeadaptivity} (Left) and { 
the figure also clearly shows sudden reduction in the time step size due to time adaptivity to cater 
for jump in $G_z$. Also, the number of time steps got reduced by an order of magnitude due to time 
adaptation of the embedded RK scheme.
 
 { In order to estimate grid convergence, simulations were performed for \numlist[
 	%list-separator       = { and },
 	list-final-separator = { and }
 	]{16;32;64;128;256} equidistant grid points in z-direction with  $u_z = \SI{80}{\cm \per \s}$. 
       The solution with the finest grid was chosen as the reference for the estimation of errors for $M_y$. 
\Cref{fig:timeadaptivity} (Right) shows the error order and the results show faster convergence rate 
than the predicted theoretical estimate.}
 }
 
% \begin{figure}[!h]
%	\centering
%	\includegraphics[width=0.75\linewidth]{TimeAdaptivity.eps}
%	\caption{Reciprocal of variable time step size as a function of time duration for 
%		$u_z= \SI{200}{\cm\per\s}$.}
%	\label{fig:timeadaptivity}
%\end{figure} 

 \begin{figure}[!h]
 	\centering
 	\includegraphics[width=1.0\linewidth]{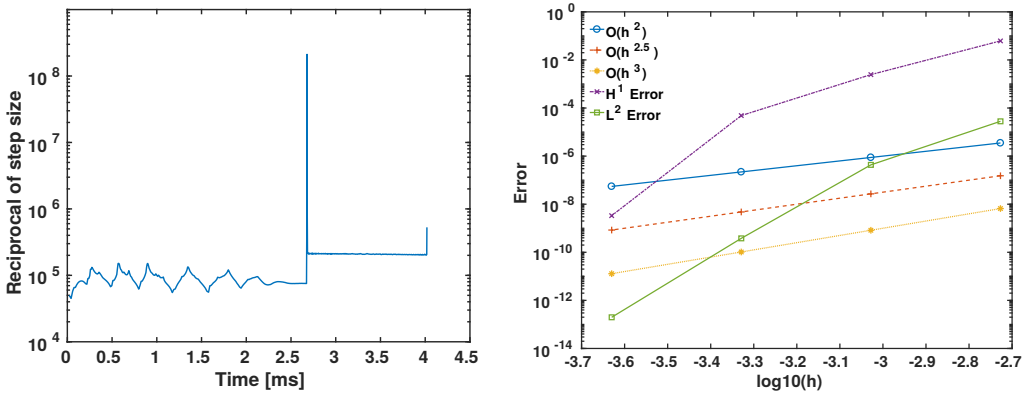}
 	\caption{Left: Reciprocal of variable time step size as a function of time duration for $u_z= \SI{200}{\cm\per\s}$. Right: Numerical errors  estimated with different grid sizes
% 		with \numlist[
% 		%list-separator       = { and },
% 		 		list-final-separator = { and }
% 		 		]{16;32;64;128;256} 
 comparing with the reference solution obtained dividing
 	the length in $z$-direction into \num{256} grid points and with $u_z=\SI{80}{\cm\per \s}$. Different error estimates are plotted in logarithmic scale as a function of logarithm of grid sizes.}
% 	finest grid from \numlist[ 
% 		%list-separator       = { and },
% 		list-final-separator = { and }
% 		]{16;32;64;128;256}  as a function of logarithm of grid length for $u_z=\SI{80}{\cm\per \s}$ }
 	\label{fig:timeadaptivity}
 \end{figure}

 \vspace{1em}
The penalty diffusion term $p_\epsilon$ with constant $\tilde \epsilon = \num{5.0E-04}$ was used to remove 
unphysical oscillations appearing in the neighborhood of inflow boundary for $M_x$ and $M_y$. Such oscillations 
for $M_y$ are shown in \Cref{fig:unphysical}. Similar oscillations were observed for $M_x$ but not for
$M_z$.

\begin{figure}[!t]
	\includegraphics[width=0.90\linewidth]{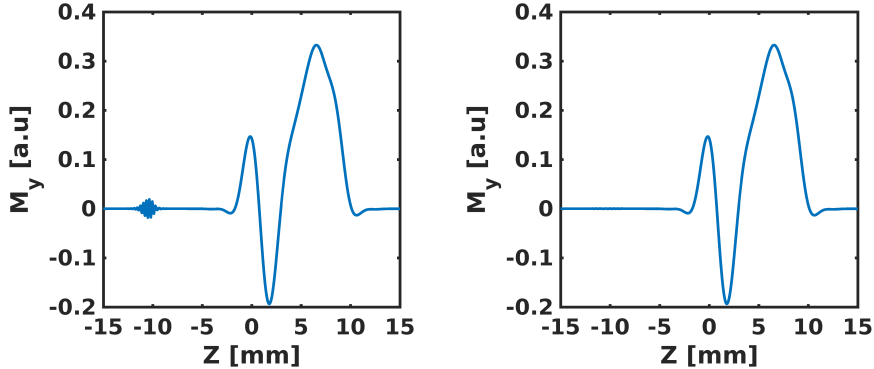}
	\caption{Simulation for $u_z=\SI{160}{\cm \per \s}$. Left: with $\epsilon=0$ results in unphysical oscillations in the neighborhood of inflow boundary 
		for $M_y$. Right: Unphysical oscillations vanish after the introduction of a small artificial 
		diffusion term with $\tilde \epsilon = \num{5.0E-4}$.} 
	\label{fig:unphysical}
\end{figure}

%\begin{figure}[!t]
%	\includegraphics[width=0.90\linewidth]{Error.eps}
%	\caption{Error as a function of step size sizes for $u_z=\SI{80}{\cm\per \s}$} 
%	\label{fig:errconv}
%\end{figure} 

%\begin{figure}[!t]
%	\includegraphics[width=0.60\linewidth]{L2error_u160.pdf}
%	\caption{} 
%	\label{fig:L2error}
%\end{figure} 

%
Finally, the results are compared with \cite{Yuan1987solution} and they show extremely good agreement. 
The effect of velocity $u_z$ in the range \SIrange{0}{200}{\cm\per\s} on magnetization is shown 
in \Cref{fig:yuanhi}.

\begin{figure}[!t]
	\includegraphics[width=0.90\linewidth]{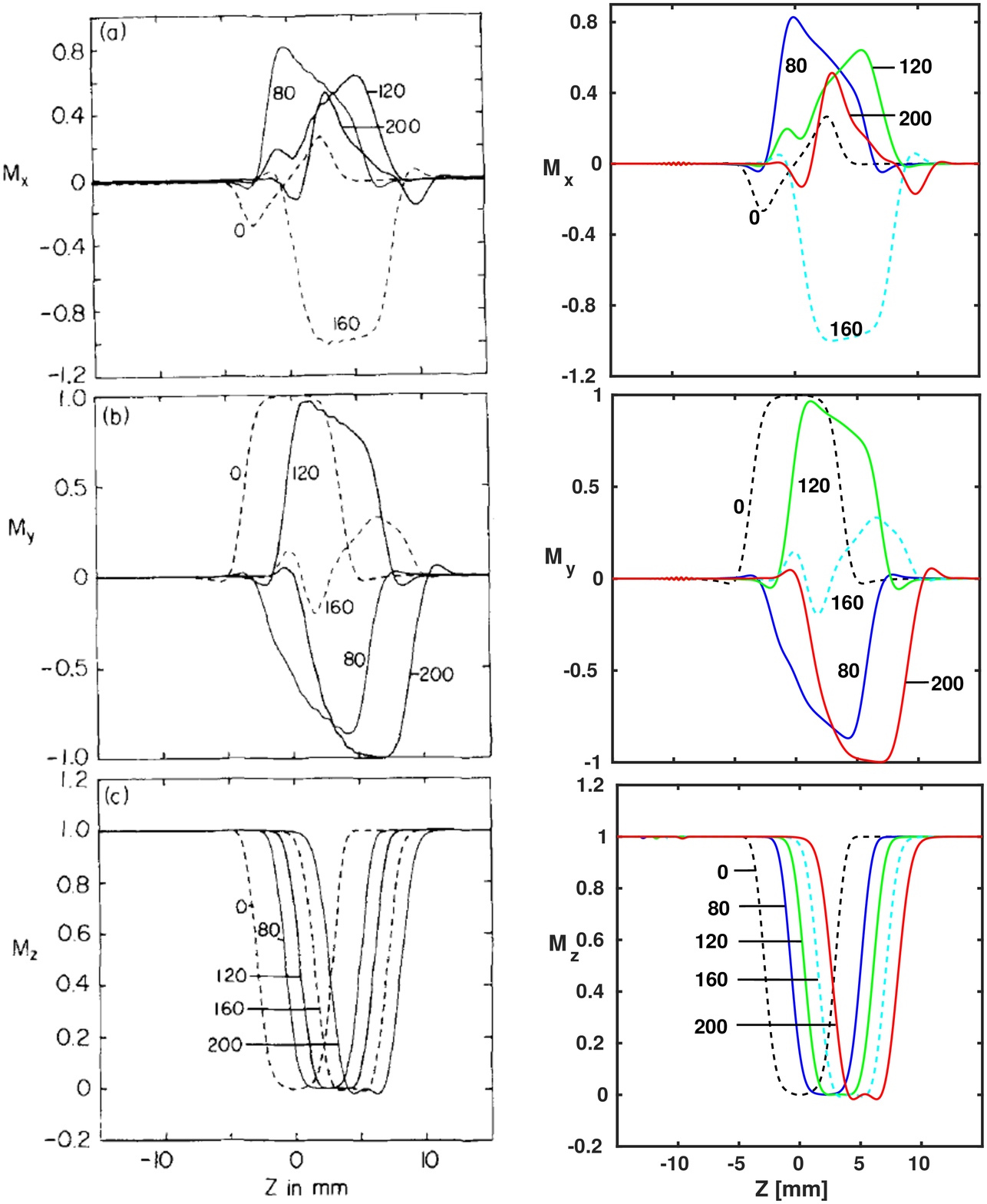}
	\caption{Simulated magnetization distributions of $M_x$, $M_y$, $M_z$ for through-plane velocity $u_z$ 
		along the positive $z$-axis in the range \SIrange{0}{200}{\cm\per\s} using dG-method (right) are 
		compared with the results in \cite{Yuan1987solution} (Left). The magnetizations were recorded at the 
		end of rewinder gradient as marked by the arrow \Cref{fig:yuanpulseseq}. The length in the slice 
		direction is in the range \SIrange{-15}{15}{\mm}.}
	\label{fig:yuanhi}
\end{figure}

\subsection{Comparison with Experiments for Through-plane Flow} \label{subsec:5.3}
The simulator was evaluated further against a laminar flow experiment in a circular tube. 
In the experiment, the flow pump was operated at different voltages to produce velocities such that 
the flow profile could be expected to be laminar i.e. Reynolds number $Re\leq 2300$ as listed in 
\Cref{table:pumpoperation}.

The flow velocities were estimated pixelwise using phase contrast MRI (PC MRI) \cite{Joseph2014}.
%with sequence parameters  TR/TE = \SI{5/4.32}{\ms}, in-plane resolution 
%= \SI[product-units = power]{1.6 x 1.6 x 6.0}{\mm}, FOV = \SI[product-units = power]{192 x 192}{\mm}, 
%number of spokes/turns = \num{7/5} and base resolution = \num{160}, nominal slice thickness = \SI{6}{\mm}, 
%$flip angle \ang{8}.
At each listed operating voltage in \Cref{table:pumpoperation}, the mean through-plane flow velocity $u_z$ 
was calculated over a pixel. The measured velocities show an unsteady pattern with a mean and a 
standard deviation as listed in the second and the third column of \Cref{table:pumpoperation}.\\
 
\begin{table}[h]
 \caption{Mean and standard deviation velocities and $Re$ based on the mean velocity for different 
          operating voltages of the flow pump at temperature \SI{16}{\celsius} (kinematic viscosity 
          $\nu = \SI{	1.1092e-02}{\cm^2\per \s}$).}
	\centering
	\begin{tabular}{|c c c c| }
		\hline
		\rowcolor{gray}
		Voltage [\si{\volt}] & Mean Velocity [\si{\mm\per\s}] &Standard Deviation 
		[\si{\mm \per \s}] & $Re$\\
		\rowcolor{lightgray}
		6 & 49.19 & 2.26 & 2217\\
		\rowcolor{lightgray}
		5 & 38.71 & 1.97 & 1744\\
		\rowcolor{lightgray}
		4 & 28.84 & 1.47 & 1300 \\
		\rowcolor{lightgray}
		3 & 18.52 & 1.04&834\\
		\hline
		%  2.5 & 13.68  & 0.75\\
		%   \hline
	\end{tabular}
	\label{table:pumpoperation}
\end{table}

\vspace{1em}

\begin{figure}[t]
 \centering
\includegraphics[width=1.0\linewidth]{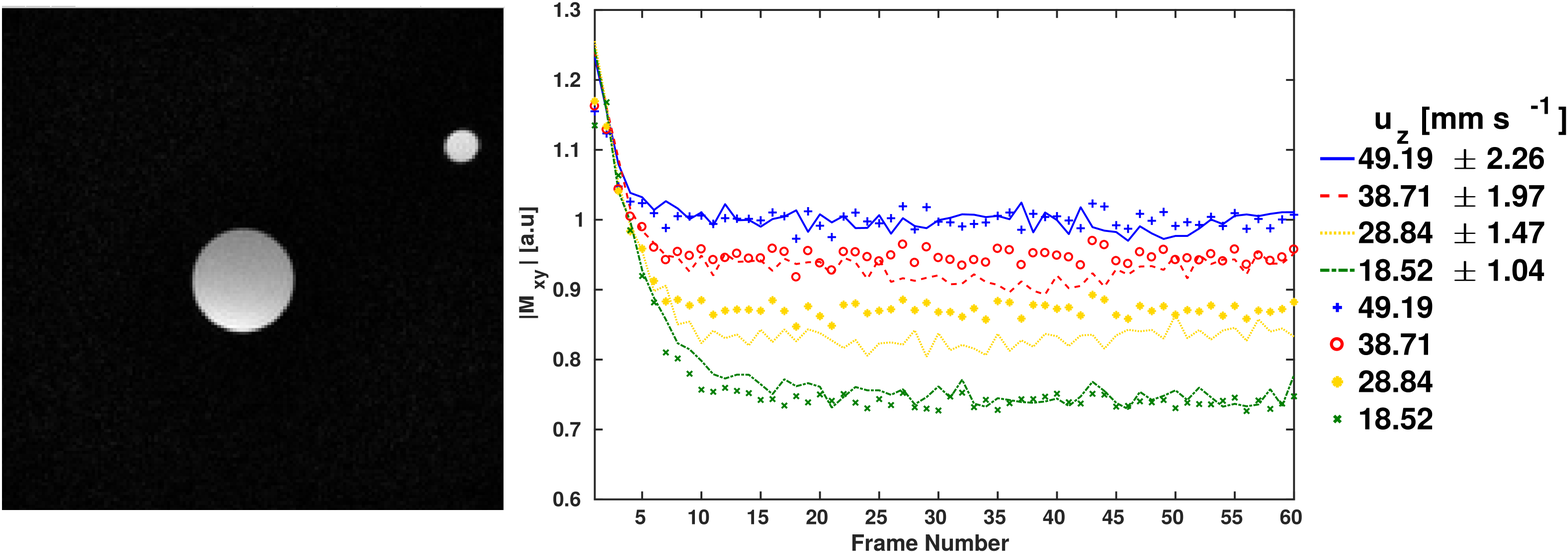}
  \caption{Left: MR image of the flow tube. Right: Experiments with different velocities are compared with the simulations.}
	\label{fig:constflowsetup}
\end{figure} 

At each specified flow velocity $u_z$, measurements were performed with the FLASH pulse sequence 
parameters TR/TE = \SI{1.96/1.22}{\ms}, flip angle = \ang{8}, and number of spokes per image frame = \num{17}. 
The resolution of one pixel in the $xy$-plane is \SI[product-units = power]{1.6 x 1.6}{\mm}
whereas the nominal slice thickness in $z$-direction is \SI{6}{\mm}. The time series of averaged magnitude 
is recorded over the same region of interest as the region of flow velocity measurement. \num{60} frames from the beginning of experiments towards dynamic equilibrium were used for the comparison with simulation.
% using \gls{pcmri} after the measurement with the radial \gls{flash} sequence. A radial \gls{pcmri} sequence 
%with was used to calculate the flow velocity. 
%Joseph et al. showed in \cite{} that the radial \gls{pcmri} sequences was proved to provide good results 
%for constant flow. 

The mean velocity $\vec u = (0, 0, u_z)^T$ was taken as input velocity in the simulation. 
A computational domain of \SI[product-units = power]{1.6 x 1.6 x 18.0}{\mm}, divided into \num{6 x 6 x 27} identical cells, was chosen for the simulation. The length of domain in the through-plane flow direction was estimated from previous simulation results \cite{Hazra2016} such that 
$\vec M_{\Gamma_{-}} = (0, 0, 1)^T$, i.e. the Dirichlet inflow boundary condition, could be satisfied. 
The initial condition was chosen as in the static case. The spatial discretization was performed using 
dG-FEM with quadratic elements ($k=2$) whereas in time the fully coupled approach with the RK5(4)-scheme 
for time discretization was applied. Again the penalty diffusion term $p_\epsilon$ with constant 
$\epsilon = \num{5.0E-04}$ was used. 

For comparison, the experimental and simulated data are normalized properly first. Due to unsteadiness of the flow, the experimental data was normalized by the average magnitude of last \num{20} frames from the measurement with the operating voltage of \SI{6}{\volt} in 
\Cref{table:pumpoperation}. The simulated data was normalized similarly. After that, they are plotted in the right part of \Cref{fig:constflowsetup}.

In spite of the unsteadiness in the flow which may be attributed to the lack of sufficient entry length, as 
discussed in \cite{Hazra2016}, the simulation and the experiment show a reasonable agreement.

Like the static phantom, GPU computing can be used for the flowing fluid as well to achieve a significant speed up 
as explained in Subsec. 4.4.3 of \cite{Hazra2016}.

\subsection{Comparison with Pulsatile Flow Experiments} \label{subsec:5.4}

The simulation method was evaluated further with a pulsatile flow laboratory experiment. Like in the 
previous case, the pulsatile velocity profile was estimated using PC MRI. A through-plane profile was 
fitted as a function of time using Matlab curve-fitting toolbox 
as shown in the left part of \Cref{fig:pulsatile}. The fitted pulsatile flow profile was used as 
input velocity in the simulation.
The effect of pulsatile flow on the evolution of magnitude from the image was studied. An experiment was performed using the pulse sequence identical to the previous experiment.

\begin{figure}[htb!]
  \centering
  \includegraphics[width=1.0\linewidth]{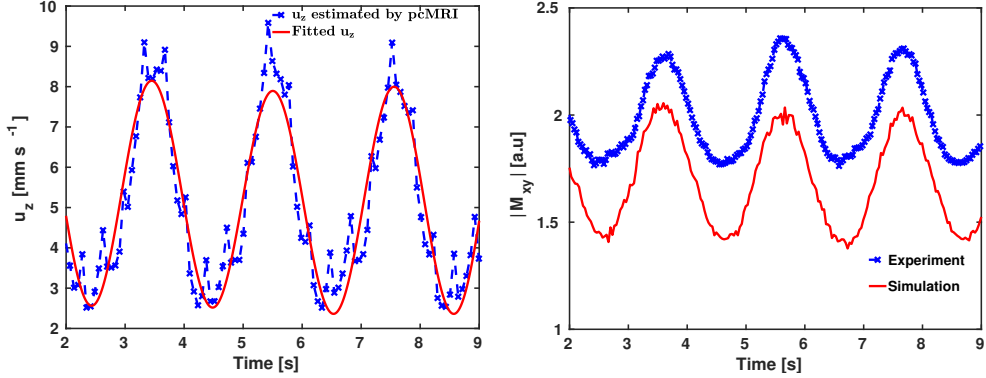}
  \caption{Left: Fitted through-plane velocity profile from PC MRI data. Right: Comparison between experimental 
           and simulation results for studying the effect of pulsatile flow.}
  \label{fig:pulsatile}
\end{figure}
\vspace{1em}

The simulation was performed with identical domain, spatial and temporal discretization, and penalty diffusion 
as used in the previous experiment. % taking the pulsatile velocity profile as the input. 
Experimental data was normalized with the magnitude data in dynamic equilibrium from the spatially 
stationary tube under identical experimental conditions. 
The simulated data was normalized by the averaged magnitude of integrated transverse magnetization of static water in dynamic equilibrium. The experiment and the simulation are compared in a state of dynamic equilibrium in \Cref{fig:pulsatile} (right).
%\begin{figure}[htb!]
%	\centering
%	\includegraphics[width=0.8\linewidth]
%	{pulsatile_thesis.png}
%	\caption{The effect of pulsatile flow on the signal can be observed 
%		here. Signal is normalized with the steady-state signal of the static water.}
%	\label{fig:SignalComparisonPulsatile}
%\end{figure}

\Cref{fig:pulsatile} shows that, although the periodicity in the experimental and the simulated result agree well, the amplitude of experimental results deviates from the simulation.
The deviation could be due to the assumption of a flow profile which depends only on time. The assumption implies that the fluid needs to move in bulk, i.e., the tube must respond simultaneously to the changing pressure 
at all positions at every specific point of time in the direction of through-plane flow such
that through-plane velocities at every position in the longitudinal direction are same, which is artificial 
and unphysical. 
%In order to fulfil the condition, .
 Nevertheless, it is a starting point to study the effect of more realistic pulsatile flows. 

\section{Summary. Outlook}

In this note, we proved the well-posedness of the {Bloch} model under the action of an incompressible 
flow field. Then we applied the discontinuous Galerkin method to the spatial semi-discretization of the
Bloch model and proved well-posedness and error estimates. The multiscale character of the problem 
basically stems from the high frequency time evolution of the magnetization part. An explicit Runge-Kutta
method together with time step adaption is applied for the temporal discretization. Alternatively,
an operator splitting between advection and magnetization can be applied. The computation can be 
strongly accelerated via GPU computing.

Magnetic resonance imaging is nowadays a very rapid process which can be done in real-time \cite{uecker2010real}.
%{ With state-of-the-art real time MRI techniques, it is possible to acquire \num{50} or more frames within a second. However even with the a highly parallelized GPU implementation it takes several orders of magnitude high time to acquire same number of frames using numerical simulation.}
Nevertheless, there are still unsolved problems such as a quantitative understanding of the mechanisms
that lead to MRI signal alterations (i.e., both enhancement and loss) when imaging flowing spins (e.g., 
in vessels or the heart) or other dynamic processes.
Here, numerical simulation with the proposed direct solver for the Bloch model can help in a better 
understanding of such dynamic processes as shown for some basic MRI experiments.
Apart from flow velocities and volumes, 
there is an increasing demand in MRI for quantitative information such as relaxation time constants. 
In future, access to both high-contrast imaging and quantitative parametric mapping by MRI is expected 
to facilitate and contribute to computer-aided diagnostic strategies.

\section*{Acknowledgement}
The first author thanks Prof. Dr. J. Frahm of the Biomedizinische NMR Forschungs GmbH (BiomedNMR) for 
suggesting the topic of MRI and many helpful discussions and Dr. D. Voit (BiomedNMR) for stimulating discussions on various occasions and numerous guidance for the experiments. 
We thank P. Schroeder for his helpful remarks. Moreover, we would also like to thank two anonymous referees for their helpful comments.

%\endgroup
%
%
%\Crefrange{fig:yuanlo}{fig:yuanhi} show an excellent agreement between the results obtained using the Leap-frog method \cite{Thomas1995,Thomas1999} in \cite{Yuan1987solution} and the results using the splitting method used in the present thesis.
%
%The plots show a shift for magnetizations along the direction of flow. The effective slice length also increases with increasing velocity. Symmetry of $M_x$ and $M_y$ break with increasing flow velocity as well. Therefore, a proper estimation of slice profile is necessary for choosing the length of computational domain in slice direction, which is elaborated and taken into consideration for comparison of simulation with experiments in \Cref{ch:ch6}.

%\section*{References}
\addcontentsline{toc}{section}{References}


%\bibliography{references}

% \RequirePackage[backend=biber, style=numeric-comp, citestyle=numeric-comp, sorting=nyt,natbib=true]{biblatex}
%\RequirePackage[style=apa,sorting=nyt,backend=biber, natbib=true]{biblatex}
%\bibliography{References/references} %Location of references.bib only for biblatex
%\bibliography{References/test.bib} %Location of references.bib only for biblatex
%\printbibliography[heading=bibintoc, title={References}]
\bibliographystyle{elsarticle-num}
\bibliography{ref}

\begin{thebibliography}{10}
\expandafter\ifx\csname url\endcsname\relax
  \def\url#1{\texttt{#1}}\fi
\expandafter\ifx\csname urlprefix\endcsname\relax\def\urlprefix{URL }\fi
\expandafter\ifx\csname href\endcsname\relax
  \def\href#1#2{#2} \def\path#1{#1}\fi

\bibitem{nishimura1996principles}
D.~G. Nishimura, Principles of magnetic resonance imaging, Stanford University,
  1996.

\bibitem{Bloch1946nuclear}
F.~Bloch, {Nuclear induction}, Physical Review 70~(7-8) (1946) 460--474.

\bibitem{Stejskal1965}
E.~O. Stejskal, J.~E. Tanner, {Spin diffusion measurements: spin echoes in the
  presence of a time-dependant field gradient}, The Journal of chemical physics
  42~(1) (1965) 288--292.

\bibitem{Lorthois2005numerical}
S.~Lorthois, J.~Stroud-Rossman, S.~Berger, L.~D. Jou, D.~Saloner, {Numerical
  simulation of magnetic resonance angiographies of an anatomically realistic
  stenotic carotid bifurcation}, Annals of Biomedical Engineering 33~(3) (2005)
  270--283.

\bibitem{Bittoun1984computer}
J.~Bittoun, J.~Taquin, M.~Sauzade, {A computer algorithm for the simulation of
  any Nuclear Magnetic Resonance (NMR) imaging method}, Magnetic Resonance
  Imaging 2~(2) (1984) 113--120.

\bibitem{Stoecker2010high}
T.~St{\"o}cker, K.~Vahedipour, D.~Pflugfelder, N.~J. Shah, {High-performance
  computing MRI simulations}, Magnetic Resonance in Medicine 64~(1) (2010)
  186--193.
\newblock \href {http://dx.doi.org/10.1002/mrm.22406}
  {\path{doi:10.1002/mrm.22406}}.

\bibitem{Hargreaves2004bloch}
B.~Hargreaves, Bloch equation simulator,
  \url{http://mrsrl.stanford.edu/~brian/blochsim/}, published January 13, 2004.

\bibitem{Benoit-Cattin2005simri}
H.~Benoit-Cattin, G.~Collewet, B.~Belaroussi, H.~Saint-Jalmes, C.~Odet, {The
  SIMRI project: A versatile and interactive MRI simulator}, Journal of
  Magnetic Resonance 173~(1) (2005) 97--115.

\bibitem{Xanthis2014mrisimul}
C.~G. Xanthis, I.~E. Venetis, A.~V. Chalkias, A.~H. Aletras, {MRISIMUL: A
  GPU-based parallel approach to MRI simulations}, IEEE Transactions on Medical
  Imaging 33~(3) (2014) 607--617.

\bibitem{Yuan1987solution}
C.~Yuan, G.~T. Gullberg, D.~L. Parker, The solution of bloch equations for
  flowing spins during a selective pulse using a finite difference method,
  Medical physics 14~(6) (1987) 914--921.

\bibitem{Jou1996calculation}
L.~D. Jou, R.~van Tyen, S.~A. Berger, D.~Saloner, {Calculation of the
  magnetization distribution for fluid flow in curved vessels}, Magnetic
  Resonance in Medicine 35~(4) (1996) 577--584.

\bibitem{Jurczuk2013computational}
K.~Jurczuk, M.~Kretowski, J.~J. Bellanger, P.~A. Eliat, H.~Saint-Jalmes,
  J.~B\'{e}zy-Wendling, {Computational modeling of MR flow imaging by the
  lattice Boltzmann method and Bloch equation}, Magnetic Resonance Imaging
  31~(7) (2013) 1163--1173.

\bibitem{ern2004theory}
A.~Ern, J.-L. Guermond, {Theory and practice of finite elements}, Vol. 159,
  Springer Science \& Business Media, 2004.

\bibitem{Di2011mathematical}
D.~A. {Di Pietro}, A.~Ern, {Mathematical aspects of discontinuous Galerkin
  methods}, Vol.~69, Springer Science \& Business Media, 2011.

\bibitem{lions2006perturbations}
J.~L. Lions, Perturbations singuli{\`e}res dans les probl{\`e}mes aux limites
  et en contr{\^o}le optimal, Vol. 323, Springer, 2006.

\bibitem{Frahm1985verfahren}
J.~Frahm, A.~Haase, D.~Matthaei, W.~H{\"a}nicke, K.-D. Merboldt, {Verfahren und
  Einrichtung zur schnellen Akquisition von Spinresonanzdaten f{\"u}r eine
  ortsaufgel{\"o}ste Untersuchung eines Objekts} (1985).

\bibitem{haase1986flash}
A.~Haase, J.~Frahm, D.~Matthaei, W.~Hanicke, K.-D. Merboldt, {FLASH imaging.
  Rapid NMR imaging using low flip-angle pulses}, Journal of Magnetic Resonance
  67~(2) (1986) 258--266.

\bibitem{deuflhard2012scientific}
P.~Deuflhard, F.~Bornemann, Scientific computing with ordinary differential
  equations, Vol.~42, Springer Science \& Business Media, 2012.

\bibitem{Hundsdorfer2003}
W.~Hundsdorfer, J.~G. Verwer, {Numerical Solution of Time-Dependent
  Advection-Diffusion-Reaction Equations}, Springer Science {\&} Business
  Media, 2003.

\bibitem{Roeloffs2015spoiling}
V.~Roeloffs, D.~Voit, J.~Frahm, Spoiling without additional gradients: Radial
  flash mri with randomized radiofrequency phases, Magnetic Resonance in
  Medicine 75~(5) (2016) 2094--2099.
\newblock \href {http://dx.doi.org/10.1002/mrm.25809}
  {\path{doi:10.1002/mrm.25809}}.

\bibitem{Carr1958}
H.~Y. Carr, {Steady-state free precession in nuclear magnetic resonance},
  Physical Review 112~(5) (1958) 1693--1701.

\bibitem{Sun2012bloch}
H.~Sun, Bloch equation simulator,
  \url{https://github.com/ismrmrd/ismrmrd-paper/blob/master/code/extern/irt/mri-rf/sun-bloch/blochCim.m},
  published Jul 22,2012.

\bibitem{Shkarin1997time}
P.~Shkarin, R.~G.~S. Spencer, {Time domain simulation of Fourier imaging by
  summation of isochromats}, International journal of imaging systems and
  technology 8 (1997) 419--426.

\bibitem{Hazra2016}
A.~Hazra, {Numerical Simulation of Bloch Equations for Dynamic Magnetic
  Resonance Imaging}, Ph.D. thesis, Institut f{\"u}r Numerische und Angewandte
  Mathematik (2016).

\bibitem{Joseph2014}
A.~Joseph, J.~T. Kowallick, K.~D. Merboldt, D.~Voit, S.~Schaetz, S.~Zhang,
  J.~M. Sohns, J.~Lotz, J.~Frahm, {Real-time flow MRI of the aorta at a
  resolution of 40 msec}, Journal of Magnetic Resonance Imaging 40~(1) (2014)
  206--213.

\bibitem{uecker2010real}
M.~Uecker, S.~Zhang, D.~Voit, A.~Karaus, K.-D. Merboldt, J.~Frahm, Real-time
  mri at a resolution of 20 ms, NMR in Biomedicine 23~(8) (2010) 986--994.

\end{thebibliography}
%\printbibliography
%\input{bibliography.tex}

%  \begin{thebibliography}{00}
% % 
% % %% \bibitem{label}
% % %% Text of bibliographic item
% % 
% %\bibitem{}
% 
%  \end{thebibliography}
\end{document}